\documentclass[11pt]{amsart}
\usepackage{amssymb,amsfonts,graphicx,amsmath,amsthm,amstext,latexsym,amsfonts}
\usepackage{graphicx,epsfig}
\usepackage{bbm}
\usepackage{t1enc}
\usepackage{mathrsfs}
\usepackage{subfig}
\usepackage{hyperref}
\usepackage{a4wide}
\usepackage{tikz}
\usetikzlibrary{intersections}
\usepackage{egothic}
\usepackage [T1] {fontenc}
\theoremstyle{plain}

\numberwithin{equation}{section}
\newtheorem{theorem}{Theorem}[section]
\newtheorem{proposition}[theorem]{Proposition}
\newtheorem{corollary}[theorem]{Corollary}
\newtheorem{lemma}[theorem]{Lemma}
\newtheorem{example}[theorem]{Example}
\newtheorem{remark}[theorem]{Remark}

\newenvironment{proofad1}{\removelastskip\par\medskip
\noindent{\textbf {Proof of Theorem \ref{mainth1}}.}
\rm}{\penalty-20\null\hfill$\blacksquare$\par\medbreak} 

\newenvironment{proofad2}{\removelastskip\par\medskip
\noindent{\textbf {Proof of Corollary \ref{nonhom}}.}
\rm}{\penalty-20\null\hfill$\blacksquare$\par\medbreak} 
\usepackage{color}
\definecolor{darkred}{rgb}{0.8,0,0}
\definecolor{darkblue}{rgb}{0,0,0.7}
\definecolor{darkgreen}{rgb}{0,0.4,0}

\newcommand{\eps}{\varepsilon}

\newcommand{\R}{{\mathbb R}}

\newcommand{\W}{{\mathcal W}}

\newcommand{\FF}{{\mathcal F}}
\newcommand{\EE}{{\mathcal E}}

\newcommand{\V}{{\mathcal V}}

\newcommand{\un}{{\rm 1\kern -2.5pt l}}
\newcommand{\tr}{{\rm Tr}}

\def\w{\mathbf{w}}
\def\u{\mathbf{u}}

\def\vv{\mathbf{v}}
\def\yy{\mathbf{y}}
\def\n{\mathbf{n}}

\def\xx{\mathbf{x}}

\def\vv{\mathbf{v}}
\def\zz{\mathbf{z}}

\def\eps{\varepsilon}
\def\R{{\mathbb R}}

\def\M{{\mathcal M}}

\def\H{{\mathcal H}}

\def\eps{\varepsilon}
\def\R{{\mathbb R}}

\def\e{{\mathcal F}}

\def\M{{\mathcal M}}

\def\H{{\mathcal H}}

\def\F{{\mathcal F}}

\def\E{{\mathbb{E}}}

\def\argmin{\mathop{{\rm argmin}}\nolimits}

\def\Tr{\mathop{{\rm Tr}}\nolimits}

\def\dv{\mathop{{\rm div}}\nolimits}
\def\curl{\mathop{{\rm curl}}\nolimits}
\def\sym{\mathop{{\rm sym}}\nolimits}

\def\span{\mathop{{\rm span}}\nolimits}
\def\u{\mathbf{u}}
\def\v{\mathbf{v}}
\def\e{\mathbf{E}}
\def\z{\mathbf{z}}
\def\v{{\bf v}}
\def\w{{\bf w}}
\def\e{{\bf e}}
\def\x{{\bf x}}

\def\Id{\mathbf{I}}

\def\wconv{\rightharpoonup}

\newcommand{\KKK}{\color{black}}

\renewcommand{\epsilon}{\varepsilon}
\newcommand{\beeq}{\begin{equation}}
\newcommand{\eneq}{\end{equation}}
\newcommand{\bear}{\begin{array}}
\newcommand{\enar}{\end{array}}
\newcommand{\bema}{\begin{displaymath}}
\newcommand{\enma}{\end{displaymath}}

\newcommand{\beea}{\begin{eqnarray}}
\newcommand{\enea}{\end{eqnarray}}

\newcommand{\om}{\Omega}

\newcommand{\bb}{\boldsymbol}

\def \qed{ \hfill \rule{6pt}{6pt}
\medskip}

\newcommand{\lab}[1]{ \label{#1} }

\def\Id{\mathbf{I}}

\def\wconv{\rightharpoonup}

\title[Linearization of incompressible elasticity]{Linearization of elasticity models for incompressible materials}
  \author{Edoardo Mainini}
\address[Edoardo Mainini]{Dipartimento di Ingegneria meccanica, energetica, gestionale e dei trasporti, 
  Universit\`a  degli studi di Genova, Via all'Opera Pia, 15 - 16145 Genova Italy.}
\email{mainini@dime.unige.it}
\urladdr{http://www.dime.unige.it/it/users/edoardo-mainini}

\author{Danilo Percivale}
\address[Danilo Percivale]{Dipartimento di Ingegneria meccanica, energetica, gestionale e dei trasporti, 
  Universit\`a  degli studi di Genova, Via all'Opera Pia, 15 - 16145 Genova Italy.}
\email{percivale@dime.unige.it}

\subjclass[2010]{49J45, 74K30, 74K35, 74R10}
\keywords{Calculus of Variations, 
  Linear Elasticity, Finite Elasticity,   
  Gamma-convergence,  Rubber-like materials}
\begin{document}
 \maketitle
\begin{abstract}
We obtain  linear elasticity  as $\Gamma$-limit   of finite elasticity under incompressibility assumption and Dirichlet boundary conditions. The result is shown for a large class of energy densities for rubber-like materials. 
\end{abstract}


\section{Introduction}
 
Organic materials require sophisticated mechanical models. They exhibit nonlinear stress-strain behaviors and are often  elastic up to large strains. Typical examples are natural rubber as well as  artificial elastic polymers with rubber-like properties. 
 These materials resist volume changes, are very compliant in shear and their shear modulus is of orders of magnitude smaller than the shear resistance of most metals. This motivate the modeling of rubber-like materials as being incompressible 
   and hyperelastic 
\cite{HA, og2, og1,  SO}, so that their phenomenological description requires the introduction of an empirical strain energy density \cite{BA, MV, SHP}.  Similar properties are observed in soft biological tissues: in many biomechanical studies blood vessels are modeled as nonlinear elastic materials that are incompressible under physiological loads  \cite{CF,HO,KS,HSMGR}.

\smallskip

From a mathematical point of view, we  consider a hyperelastic body that occupies a bounded open region $\Omega\subset \mathbb R^3$ in its reference configuration. In presence of a body force field $\mathbf g: \Omega\to\mathbb R^3$, the energy of the system is given by  the stored elastic energy and the contribution of the external forces  \[
 \int_\Omega\mathcal W^I(x,\nabla\mathbf y(x))\,dx-\int_\Omega (\mathbf y(x)-x)\cdot\mathbf g(x)\,dx.
 \]
 Here, $\mathbf y:\Omega\to\mathbb R^3$ denotes the deformation field, $\nabla \mathbf y$ denotes the deformation gradient and $\mathcal W^I$ is the incompressible elastic energy density.  $\mathcal W^I$ is assumed to be frame indifferent and minimized at the identity with $\mathcal W^I(x,\mathbf I)=0$, so that without an external load $\mathbf y(x)=x$ is a minimizer of the total energy corresponding to the stress-free configuration $\Omega$. In order to take into account that the body is incompressible, $\mathcal W^I(x,\mathbf F)=+\infty$ whenever $\det \mathbf F\not =1$.
 
 \smallskip
 
 A common approach in the study of rubber-like materials is to consider a stored energy density $\mathcal W$ which is defined in the compressible range,  the kinematic constraint $\det \mathbf F=1$ being relaxed to a volumetric penalization: a typical  expression of $\mathcal W$ is given by the usual isochoric-volumetric form 
 \begin{equation}\label{isovol}
\mathcal W(x,\mathbf F): =\mathcal W_{iso}(x, (\det \mathbf F)^{-1/3}\mathbf F)+\mathcal W_{vol}(\det \mathbf F)
\end{equation}
where  $x\in\om$ and $\det\mathbf F> 0$ (extended to $+\infty$ if $\det\mathbf F\le 0$). Here, the nonnegative function $\mathcal W_{iso}(x,\mathbf F_*)$ is defined for every  $\mathbf F_*$ such that $\det\mathbf F_*=1$ and satisfies $\mathcal W_{iso}(x,\mathbf I)=0$. Moreover,    $\mathcal W_{vol}(t)\ge 0$ for every $t> 0$ and $ \mathcal W_{vol}(1)=0$. 
In fact, we shall first choose a compressible energy density $\mathcal W$, for instance in the form \eqref{isovol},  requiring frame indifference and other suitable regularity conditions that will be introduced in Section \ref{sectmain}. Then, we shall  define the incompressible energy density $\mathcal W^I$ by setting $\mathcal W^I(x,\mathbf F)=\mathcal W(x,\mathbf F)$ if $\det\mathbf F=1$ and $\mathcal W^I(x,\mathbf F)=+\infty $ if $\det\mathbf F\neq1$. 


\smallskip

The simplest model is the homogeneous  Neo-Hookean solid:  the energy density  is of the form
\eqref{isovol} with 
\begin{equation}\label{neo}
\W_{iso} (\mathbf F_*):=
\mu\,((\mathrm{Tr}(\mathbf F_*^T\mathbf F_*) -3),
\end{equation}
where $\det \mathbf F_*=1$
 and the shear modulus $\mu$ is determined experimentally: this model fits  material behaviors with sufficient accuracy under moderate straining while, at higher strains, it can be  replaced by the more general Ogden model, namely 
\begin{equation}\lab{ogden}
\mathcal W_{iso}(\mathbf F_*):=\sum_{p=1}^{N}\frac{\mu_p}{\alpha_p}(\mathrm{Tr}((\mathbf F_*^T\mathbf F_*)^{\alpha_p/2})-3)
\end{equation}
where  $N, \mu_p, \alpha_p$ are material constants. For particular values of the material constants the Ogden model reduces to either the Neo-Hookean solid ($N=1,\ \alpha_1=2$) or the so called Mooney-Rivlin material ($N=2,\ \alpha_1=2,\ \alpha_2=-2$) which is often applied to model incompressible biological tissue, see \cite{ML1,ML2}.
Another phenomenological material model, motivated for simulating the mechanical behavior of carbon-black filled rubber and for its important applications in the manufacture of automotive tyres, has been introduced by Yeoh, see \cite{Y,Y2}: the isochoric part of the strain energy density is given by
\begin{equation}\lab{yeoh}
\mathcal W_{iso}(\mathbf F_*):= \sum_{k=1}^3 c_k ((\mathrm{Tr}(\mathbf F_*^T\mathbf F_*)-3)^k
\end{equation}
where $c_k,\ k=1,2,3$ are material constants. 
\KKK
For a complete description of the main properties of such energy densities and other models, we refer to the classical monographs such as \cite{C, HA, OGDEN} or to the reviews in \cite{AHS,  BA, MV}, see also \cite{AB, KAN, SHP}. 

\smallskip

Let us now introduce the linearization.
If  $h > 0$ is an adimensional parameter, we scale the body force field by taking $\mathbf g:= h\mathbf f$ and set  $\mathbf y(x):=x+h \mathbf \v(x)$.  The resulting total energy is
\[
\mathcal E_h(\v):=\displaystyle \int_\om \W^I(x, \mathbf I+h\nabla \v)\,dx-h^2\int_\om\mathbf f\cdot\mathbf v\,dx\]
and it seems meaningful to ask what is the correct scaling of energies $\mathcal E_h$ as $h\to 0^+$. Roughly speaking in the spirit of \cite{ABP} (see also \cite{LM, PT1, PTplate, PT2, PT4, PPG}) we will show that, under suitable boundary conditions, 
 $$ \inf\mathcal E_h= h^2\min\mathcal E_0 +o(h^2),$$
 where 
\[ \mathcal E_0(\v):=\left\{\begin{array}{ll} \displaystyle\frac12\int_\om \mathbb E(\v)\, D^2\mathcal W(x,\mathbf I)\,\mathbb E(\v)\,dx -\int_\om\mathbf f\cdot\mathbf v\,dx\quad &\hbox{if }\ \dv\v=0\ \hbox{a.e. in}\ \om\\
&\\
+\infty\quad &\hbox{otherwise}.
\end{array}\right.
 \]
 The quadratic form appearing in the expression of $\mathcal E_0$ features  the infinitesimal strain tensor  $\E(\vv):= \tfrac12(\nabla\v+\nabla \v^T)$ and can be obtained by a formal Taylor expansion of $\mathcal W$ around the identity matrix $\mathbf I$, with $D^2$ denoting the Hessian of $\mathcal W(x,\cdot)$. We stress that purely volumetric  perturbations of $\mathcal W$ do not affect $\mathcal W^I$ nor  $\mathbb E(\v)\, D^2\mathcal W(x,\mathbf I)\,\mathbb E(\v)$, due to the divergence-free condition.
 Moreover, we will prove that if $$\mathcal E_h(\v_h)-\inf \mathcal E_h=o(h^2),$$ then $$\v_h\wconv \v_0\in \argmin\mathcal E_0$$ in the weak topology of a suitable Sobolev space.
 Since the compliance in shear of rubber-like materials and the strong nonlinearity of their stress-strain behavior even at modest strain do not allow  to suppose that  small strains correspond to small loads,  it must be clarified  that { \it it would not be reasonable to assume} that either $h\v$ or $h \mathbb E(\v)$ are small in any sense. 
 Anyhow we highlight that linearized models may provide a good approximation that fits experimental data, see for instance \cite{LZWK}.
 
 \smallskip

 
From the viewpoint of the Calculus of Variations, derivation of linearized elasticity from finite elasticity has a long history that started in \cite{DMPN},
 where $\Gamma$ convergence and convergence of minimizers of the associated Dirichlet boundary value problems are proven in the compressible case (see also \cite{ABK}, \cite{ADMDS}, \cite{ADMLP},  \cite{MPTJOTA}, \cite{MPTARMA}  for more recent results).
In this paper we show how these results can be extended to the incompressible case, i.e., assuming the constraint $\det\mathbf \nabla \mathbf y=1$ on admissible deformations fields. 
It is well-known that such a constraint poses some challenges to the  $\Gamma$-convergence analysis 
(see for instance the derivation of a two-dimensional model for elastic plates in \cite{CD}). 
  Indeed, some novel approach (that we develop in Lemma \ref{reynolds}) is required for the construction of recovery sequences, due to the necessity of recovering the linearized incompressibility constraint $\dv\v=0$ with a sequence $\v_h$ satisfying $\det(\mathbf I+h\nabla\v_h)=1$ a.e. in $\om$. Moreover,   a different strategy is also needed to ensure that the whole sequence $(\v_h)$ and $\v$ satisfy the same Dirichlet condition. To this end the crucial point consists in analyzing  vector potentials:  we show in Lemma \ref{w=0} that if $\v\in H^1(\om,\mathbb R^3)$, $\mathrm{div}\,\v=0$ in $\om$ and $\v=0$ on $\Gamma\subset\partial\om$, then, under suitable topological assumptions  (see conditions \eqref{OMEGA}-\eqref{Gamma}), there exists $\w\in H^2(\om,\mathbb R^3)$ such that $\curl \w=\v$ in $\om$ and $\w=0$ on $\Gamma$. 
  Taking advantage of this result, the construction of the recovery sequence
relies on a careful  approximation of $\w$ thus outflanking the constraint $\dv \v=0$.

 \subsection*{Plan of the paper} In Section 2 we state the main result. Its proof requires the analysis of vector potentials in Sobolev spaces, which is the object of section 3. In section 4, we develop suitable approximation results that are used for the construction of the recovery sequence in the proof of the main theorem, which is instead contained in Section 5. 
 
\section{Main result}\label{sectmain}

In this section we introduce the basic notation and all the assumptions of our theory. Then, we state the main result. 

\subsection*{Assumptions on the reference configuration}
Concerning the reference configuration $\om\subset \R^3$, we assume that
\begin{equation}\label{OMEGA}\begin{array}{ll}
& i)\  \om \hbox{ is a bounded,  simply connected open set},\\
&\\
&ii)\  \partial\om\ \hbox{is a connected} \ C^{3}\   \hbox{manifold}  
\end{array}
\end{equation}
and we let $ \n\in C^{2}(\partial\Omega)$ denote  its outward unit normal vector. { We will prescribe a Dirichlet boundary condition on a  subset $\Gamma$ of $\partial\om$.    Letting $\partial\Gamma$ denote the relative boundary of $\Gamma$ in $\partial\om$ and letting $\mathcal H^2$ denote the two-dimensional Hausdorff measure,  we  assume that 
 \begin{equation}\begin{array}{ll}\lab{Gamma}
&i)  \hbox{ $\Gamma$ is a  closed  subset of $\partial\om$  and
 $\mathcal H^2(\Gamma)>0$},  \\
&\\
&ii)\
 \hbox{either}\quad \partial\Gamma=\emptyset\quad \hbox{ or }\quad \partial\Gamma \ \hbox{is a}\  C^{3}
 \hbox{ one-dimensional  submanifold of}\ \partial\om.
\end{array}
\end{equation}}

Through the proofs we will also suppose, without loss of generality, that $\Gamma$ is connected. Indeed, all the arguments that we shall develop can be extended to the non connected case  by considering each connected component. Besides, in case $\partial \Gamma=\emptyset$, it is possible to assume that $\partial\om$ is $C^{2,1}$ only (see Remark \ref{C21} later on).

\subsection*{Some notation} 
 For a Sobolev vector field $\u\in W^{1,r}(\om,\mathbb R^3)$, $r\ge 1$,  conditions like $\u=0$ on $\Gamma$ are always understood in the sense of traces. Moreover, we shall often use the decomposition in normal and tangential part at the boundary $\u=(\u\cdot\n)\,\n+\n\wedge(\u\wedge\n)$, where $\wedge$ denotes the cross product. Bold letters will be used in general for vector fields.

\subsection*{Assumptions on the elastic energy density}
Let $\W : \om \times \mathbb R^{3 \times 3} \to [0, +\infty ]$ be  ${\mathcal L}^3\! \times\! {\mathcal B}^{9} $- measurable. 
We assume that  $\W$ is frame indifferent and minimized at the identity, i.e.
%

%
\beeq \lab{framind}\tag{$\bb{\mathcal W1}$} \W(x, \mathbf R\mathbf F)=\W(x, \mathbf F) \qquad \forall \, \mathbf\! \mathbf R\!\in\! SO(3), \quad \forall\, \mathbf F\in \mathbb R^{3 \times 3},\ \quad  \hbox{ for a.e.}\ x\in \om\,,
\eneq
\beeq \lab{Z1}\tag{$\bb{\mathcal W2}$}
\min \mathcal W=\KKK	\W(x,\mathbf I)=0\ \qquad \hbox{for a.e. }x \in \om  \,,
\eneq
Moreover, we assume that $\W(x,\cdot)$ is $C^2$ in a neighbor of rotations (with gradient and Hessian denoted by $D$ and $D^2$), i.e.,
\beeq\begin{array}{ll}\lab{reg}\tag{$\bb{\mathcal W3}$} &   \hbox{there exists a neighborhood} \ \mathcal U \  \hbox{of}\ SO(3) \hbox{ s.t., for a.e. $x\in\om$,  }
\mathcal W(x,\cdot)\in C^{2}(\mathcal U),\\
& \hbox{with a modulus of continuity of $D^2\mathcal W(x,\cdot)$ that does not depend on $x$}.\\
&\hbox{Moreover,  there exists } K>0\ \hbox{such that} \ |D^2 \mathcal W(x,\mathbf I)|\le K\,\,\,\hbox{for a.e. $ x\in\om$}.
\end{array}
\eneq
The coercivity of $\W$ is described by the following property: there exists $C>0$ and $p\in(1,2]$ such that
\beeq \lab{coerc}\tag{$\bb{\mathcal W4}$}
\begin{array}{ll}
\W(x,\mathbf F)\ge  C \ g_{p}(d(\mathbf F, SO(3)))\qquad
 \forall\, \mathbf F\in \mathbb R^{3 \times 3},\quad \hbox{for a.e.}\ x\in \om,
\end{array}
\eneq
where $g_p:[0,+\infty)\to\mathbb R$ is the convex function defined by
\beeq\lab{gp}
g_{p}(t)=\left\{\begin{array}{ll} \!\! t^{2}\quad &\hbox{if}\ 0\le t\le 1\\
&\\
\!\! \displaystyle \frac{2t^{p}}{p}-\frac{2}{p}+1\quad &\hbox{if}\ t\ge 1,\\
\end{array}\right.
\eneq


Concerning the latter assumption, we  refer to \cite{ADMDS} for a discussion about the growth properties of  energy densities of the form \eqref{isovol}-\eqref{ogden}: for certain ranges of the parameters therein and suitable choice of $\mathcal W_{vol}$, they exhibit a quadratic growth for small deformation gradients  and a $p$-growth, $1<p\le2$, for large deformation gradients, in particular they satisfy all the above assumptions. Indeed, taking the model choice $\mathcal W_{vol}(t)=c(t^2-1-2\log t)$, $c>0$, it is shown in \cite{ADMDS} that  the Neo-Hookean energy density \eqref{isovol}-\eqref{neo}  satisfies \eqref{coerc} for some $p\in(1,2)$, and in fact for $p=2$ when restricting to  $\det\mathbf F=1$ (then the same holds  for the Yeoh model \eqref{isovol}-\eqref{yeoh} with positive $c_i$'s); similarly,  the general Ogden model \eqref{isovol}-\ref{ogden}  has a $p$-growth with $p\in(1,2)$ if $\mu_i>0$ and $0<\alpha_i<3$ for all $i=1,\ldots,N$ with $\alpha_i> 6/5$ for at least one $i$. We stress that the Ogden model can exhibit a less than quadratic growth even in the incompressibility regime $\det\mathbf F=1$, as can be checked for instance by choosing  $N=1$, $\alpha_1=3/2$ and $\mathbf F=\mathrm{diag}(\lambda^{-2},\lambda,\lambda)$ therein.       
\KKK

Let us discuss some first consequences of the above assumptions on $\mathcal W$.
Since $\mathcal W\ge 0$,  assumptions
 \eqref{framind}  and \eqref{Z1} yield  
 \begin{equation}\label{smooth0}\W(x,\mathbf R)\!=\!0,\ D \W(x,\mathbf R)\!=\!0 \quad \forall \mathbf R\in SO(3),\quad \mbox{for a.e. $x\in\om$}\end{equation}
 and in particular
 the reference configuration
has zero energy and
is stress free. 
Due to frame indifference there exists a function $\mathcal V$ such that 
\begin{equation*}
\W(x,\mathbf F)=\V(x,\textstyle{\frac{1}{2}}( \mathbf F^T \mathbf F - \mathbf I))\,,
\qquad
\ \forall\, \mathbf F\in \mathbb R^{3\times 3},\  \hbox{ for a.e. }x\in \om,
\end{equation*}
so by setting $\mathbf F=\Id +h\mathbf B$, where $h> 0$ is an adimensional small parameter, we have
\begin{equation*}
h^{-2}\mathcal W(x,\Id+h\mathbf B)=h^{-2}\V(x,h\,{\rm sym}\mathbf B+h^{2}\mathbf B^{T}\mathbf B),
\end{equation*}
where $\mathrm{sym}\mathbf B:=\frac12(\mathbf B^T+\mathbf B)$, and then
we get via Taylor's expansion
\[\displaystyle 
\lim_{h\to 0} h^{-2}\mathcal W(x,\mathbf I+h\mathbf B)=
\frac{1}{2} \,{\rm sym} \mathbf  B\, D^2\V (x, \mathbf 0) \ {\rm sym}\mathbf  B=\frac12\, \mathbf B^T D^2\mathcal W(x,\mathbf I)\,\mathbf B.
\]
Hence, \eqref{coerc}  implies that for a.e. $x\in \om$
\begin{equation}\label{ellipticity}
\frac12\, \mathbf B^T D^2\mathcal W(x,\mathbf I)\,\mathbf B=\frac12\mathrm{sym}\mathbf B\,D^2\mathcal W(x,\mathbf I)\,\mathrm{sym}\mathbf B
	\, \geq \, \frac{C}{4} \, |\hbox{\rm sym}\,\mathbf  B|^2
	\quad \forall \  \mathbf  B\in \mathbb R^{3 \times 3}.
\end{equation}
Here and through the paper, for a matrix $\mathbf B\in\mathbb R^{3\times3}$, we denote $|\mathbf B|:=\sqrt{\mathrm{Tr}(\mathbf B^T\mathbf B})$.
%

\subsection*{Incompressibility}
We assume that the  material is  incompressible. This is done by introducing the incompressible elastic energy density $\W^I$ as
\[
\mathcal W^{I}(x, \mathbf F):=\left\{\begin{array}{ll} \mathcal W(x,\mathbf F)\quad &\hbox{if} \det \mathbf F=1\\
&\\
 +\infty\quad &\hbox{otherwise}\end{array}\right.
\]


\subsection*{External forces}
We introduce  a  body force field $\mathbf f\in L^{\frac{3p}{4p-3}}(\Omega,\R^3)$, where $p$ is such that \eqref{coerc} holds. The corresponding contribution to the energy is given by the following functional, defined 
 for  $\v\in W^{1,p}(\om,\mathbb R^3)$
\begin{equation}\label{external}
\mathcal L(\v):=\int_\om\mathbf f\cdot\v\,dx.
\end{equation}
 By the Sobolev embedding $W^{1,p}(\om,\mathbb R^3)\hookrightarrow L^{\frac{3p}{3-p}}(\om,\mathbb R^3)$, $\mathcal L$ is a continuous functional on $W^{1,p}(\om,\mathbb R^3)$ and $|\mathcal L(\v)|\le C_{\mathcal L}\|\v\|_{W^{1,p}(\om,\mathbb R^3)}$ holds for a suitable constant $C_{\mathcal L}$ that depends on $\om$ and $\mathbf f$.

\subsection*{Rescaled energies}
 The functional representing the scaled total energy is denoted by $\mathcal F_h^{I}: W^{1,p}(\Omega,\R^3)\to \R\cup\{+\infty\} $ and defined as follows
\begin{equation*}
\label{nonlinear}
\displaystyle \mathcal F_h^{I}(\v):=\frac1{h^2}\int_\om\mathcal W^I(x,\mathbf I+h\nabla\v)\,dx
-\mathcal L(\v).
\end{equation*}

%

\subsection*{Linearized functional}
In this paper we are interested in the  asymptotic behavior as $h\downarrow 0_+$
of functionals $\mathcal F_h^I$
and to this aim we
introduce the limit energy functional $\mathcal F^{I}: W^{1,p}(\Omega,\R^3)\to \R\cup \{+\infty\}$ defined by
\begin{equation}
\label{DTfunc}
\F^{I}(\v): =\left\{\begin{array}{ll}
\displaystyle
\frac12\int_\om \mathbb E(\v)D^2\mathcal W(x,\mathbf I)\,\mathbb E(\v)\,dx 
 -\mathcal L(\mathbf v)\quad &\hbox{if} \ \v\in H^1_{\mathrm{div}}(\om,\mathbb R^3)\\
 &\\
 \ +\infty\quad &\hbox{otherwise in} \ W^{1,p}(\om,\mathbb R^3)
 \end{array}
 \right.
\end{equation}
  where $\E(\vv)\!:=\! \hbox{sym }   \!\nabla\vv$ denotes the infinitesimal strain tensor field associated to the displacement field $\v$, and 
where $H^1_{\mathrm{div}}(\om,\mathbb R^3)$ is the set of $H^1(\om,\mathbb R^3)$ vector fields whose divergence vanishes a.e. in $\om$.

{Since we  work with the incompressible energy density $\mathcal W^I$, we stress that if the function $\mathcal W$ is replaced by any other $\widetilde {\mathcal W}$ 
 such that assumptions \eqref{framind}, \eqref{Z1}, \eqref{reg},    
\eqref{coerc} are satisfied and $\widetilde{\mathcal W}(x,\mathbf F)= {\mathcal W}(x,\mathbf F)$ as soon as $\det \mathbf F=1$, then  this does not affect functional $\mathcal F^I$. Indeed, if $\mathrm{Tr}\,\mathbf B=0$ then $\det(\exp(h\mathbf B))=\exp(h\tr\, \mathbf B)=1$, so that by Taylor's expansion we have
\begin{equation*}\begin{aligned}
&\displaystyle \frac12\,\mathrm{sym}\mathbf B\,D^2\widetilde{\mathcal W}(x,\mathbf I)\,\mathrm{sym}\mathbf B= \lim_{h\to 0}h^{-2}\widetilde {\mathcal W}(x, \mathbf I+h\mathbf B)=\lim_{h\to 0}h^{-2}\widetilde {\mathcal W}(x, \mathbf I+h\mathbf B+ o(h))\\
&\qquad\displaystyle=\lim_{h\to 0}h^{-2}\widetilde {\mathcal W}(x, \exp(h\mathbf B))=\lim_{h\to 0}h^{-2} \mathcal W^I(x, \exp(h\mathbf B))=\lim_{h\to 0}h^{-2} \mathcal W(x, \exp(h\mathbf B))\\
&\qquad\displaystyle= \lim_{h\to 0}h^{-2} \mathcal W(x,\mathbf I+h\mathbf B+ o(h))= \lim_{h\to 0}h^{-2} \mathcal W(x,\mathbf I+h\mathbf B)= \frac12\,\mathrm{sym}\mathbf B\,D^2\mathcal W(x,\mathbf I)\,\mathrm{sym}\mathbf B.
\end{aligned}
\end{equation*}
For instance, if the function $\mathcal W$ 
  is in the form \eqref{isovol}, then  $\mathcal W_{vol}$ can be arbitrarily replaced as soon as the  assumptions \eqref{framind}, \eqref{Z1}, \eqref{reg},    
\eqref{coerc} are matched.}



\subsection*{Statement of the main result}
In order to  prescribe a Dirichlet boundary condition on $\Gamma$  we define $\mathcal G^{I}: W^{1,p}(\om,\mathbb R^{3})\to \mathbb R\cup\{+\infty\}$ and  $\mathcal G_{h}^{I}: W^{1,p}(\om,\mathbb R^{3})\to \mathbb R\cup\{+\infty\}$  as
  \begin{equation*}
\mathcal G_{h}^{I}(\v):=\left\{\begin{array}{ll} \F_{h}^{I}(\v)\quad &\hbox{if}\ \v=0\ \hbox{on}\ \Gamma\\
&\\
+\infty\quad &\hbox{otherwise},
\end{array}\right.
\end{equation*}
 \begin{equation*}
\mathcal G^{I}(\v):=\left\{\begin{array}{ll} \F^{I}(\v)\quad &\hbox{if}\  \v=0\ \hbox{on}\ \Gamma\\
&\\
+\infty\quad &\hbox{otherwise}.
\end{array}\right.
\end{equation*}
We are ready for the statement of the main result
\begin{theorem}
\label{mainth1}
Assume \eqref{OMEGA}, \eqref{Gamma}, \eqref{framind}, \eqref{Z1}, \eqref{reg},    
\eqref{coerc}. 
Then for every vanishing sequence $(h_j)_{j\in\mathbb N}$  of strictly positive real numbers
we have
\beeq\lab{convmin}
\inf_{W^{1,p}(\om,\mathbb R^3)}\mathcal G^{I}_{h_{j}} \in \mathbb R. \eneq
If $(\mathbf v_j)_{j\in\mathbb N}\subset {W^{1,p}(\om,\mathbb R^3)}$ is a sequence such that $\v_j=0$ on $\Gamma$ for any $j\in\mathbb N$ and such that
\beeq \label{assinf}  \lim_{j\to +\infty}\left( \mathcal G^{I}_{h_{j}}(\v_{j})-\inf_{W^{1,p}(\om,\mathbb R^3)}\mathcal G^{I}_{h_{j}}\right)= 0,
\eneq
then we have as $j\to+\infty$ 
\[ \v_{j}\wconv \v_* \ \hbox{weakly in}\ W^{1,p}(\Omega,\R^3),
\]
where $\v_*$ is the unique minimizer of $\mathcal G^I$ over $W^{1,p}(\om,\mathbb R^3)$,
and
\begin{equation*}   \mathcal G^{I}_{h_{j}}(\v_{j})\to \mathcal G^{I}(\v_*),\qquad
\inf_{W^{1,p}(\om,\mathbb R^3)}\mathcal G^{I}_{h_{j}}\to \min_{W^{1,p}(\om,\mathbb R^3)}\mathcal G^{I}.
\end{equation*}
\end{theorem}
We remark that  under our assumptions, functionals $\mathcal G^I_{h}$ do not have minimizers in general.
In case they have,
 then it is  possible to substitute assumption \eqref{assinf} with $\mathbf v_j\in \argmin \mathcal G^I_{h_j}$ and convergence of minimizers is deduced. 

\subsection*{Nonhomogeneous boundary conditions}
It is worth noticing that Theorem \ref{mainth1} works only when we assume homogeneous boundary conditions so it is quite natural to ask what happens in a more general case. Unfortunately the proof cannot be extended to the non homogeneous case as well, nevertheless this difficulty can be in some sense circumvented as follows. 
Fix $ \overline\v\in W^{1,\infty}(\om,\mathbb R^3)$ such that $\dv\overline\v=0\,$ a.e. in $\om$ and define for $h>0$
\begin{equation}\lab{Gvbar}
\overline{\mathcal G}^I(\v):=\left\{\begin{array}{ll}
 \mathcal F^I(\v)\quad &\hbox{if}\ \v=\overline\v\ \hbox{on}\ \Gamma\\
&\\
 +\infty\quad &\hbox{otherwise},\\
\end{array}\right.
\end{equation}
\begin{equation*}\lab{tilde2Gh}
\widetilde{\mathcal G}_h^I(\v):= \mathcal G_h^I(\v)+\int_\om \mathbb E(\overline\v) D^2\mathcal W(x, \mathbf I)\mathbb E(\v)\,dx+ \overline{\mathcal G}^I(\overline\v)
\end{equation*}
and 
\begin{equation}\label{lastlabel}
\widetilde{\mathcal G}^I(\v):= \mathcal G^I(\v)+\int_\om \mathbb E(\overline\v) D^2\mathcal W(x, \mathbf I)\mathbb E(\v)\,dx+\overline{\mathcal G}^I(\overline\v).
\end{equation}
A slight modification of the proof of Theorem \ref{mainth1} gives the following
\begin{corollary}\lab{nonhom}
Assume \eqref{OMEGA}, \eqref{Gamma}, \eqref{framind}, \eqref{Z1}, \eqref{reg},    
\eqref{coerc} and let $\overline\v\in W^{1,\infty}(\om,\mathbb R^3)$ such that $\dv\overline\v=0$ in $\om$.
Then for every vanishing sequence $(h_j)_{j\in\mathbb N}$  of strictly positive real numbers
we have
\beeq\lab{convmin2}
\inf_{W^{1,p}(\om,\mathbb R^3)}\widetilde{\mathcal G}^{I}_{h_{j}} \in \mathbb R.
 \eneq
If $(\mathbf v_j)_{j\in\mathbb N}\subset {W^{1,p}(\om,\mathbb R^3)}$ is a sequence such that $\v_j=0$ on $\Gamma$ for any $j\in\mathbb N$ and such that
\beeq \label{assinf2} 
 \lim_{j\to +\infty}\left( \widetilde{\mathcal G}^{I}_{h_{j}}(\v_{j})-\inf_{W^{1,p}(\om,\mathbb R^3)}\widetilde{\mathcal G}^{I}_{h_{j}}\right)= 0,
\eneq
then we have as $j\to+\infty$
\[
\v_{j}\wconv \v_{0} \ \hbox{weakly in}\ W^{1,p}(\Omega,\R^3),
\]
where $\v_0$ is the unique minimizer of $\widetilde{\mathcal G}^{I}$ over $W^{1,p}(\om,\mathbb R^3)$,
and
\[
 \widetilde{\mathcal G}^{I}_{h_{j}}(\v_{j})\to \widetilde{\mathcal G}^{I}(\v_0),\qquad 
 \inf_{W^{1,p}(\om,\mathbb R^3)}\widetilde{\mathcal G}^{I}_{h_j}
 \to\min_{W^{1,p}(\om,\mathbb R^3)}\widetilde{\mathcal G}^I.
\]
Moreover,
$\v_0+\overline\v$ is the unique minimizer of $\overline{\mathcal G}^I$ over $W^{1,p}(\om,\mathbb R^3)$ and   $\widetilde{\mathcal G}^{I}(\v_0)=\overline{\mathcal G}^I(\v_0+\overline\v)$.
\end{corollary}

\section{Preliminary results on vector potentials}\label{Sectionpotential}

Given a divergence-free deformation field $\mathbf v$, it will often be useful to work with a vector potential $\mathbf w$, such that $\mathrm{curl}\mathbf w=\mathbf v$. The next results gather several properties of vector potentials of deformation fields satisfying suitable Dirichlet boundary conditions. We start by recalling a result that is found for instance in  \cite{ABDG}.

\begin{lemma}[$\mbox{\cite[Corollary 2.15]{ABDG}}$]\label{divrot}  Assume \eqref{OMEGA}  and let $\w\in L^{2}(\om,\mathbb R^{3})$ be such that
$\mathrm{curl}\ \w\in H^ {1}(\om,\mathbb R^{3})$,  $\mathrm{div}\ \w\in H^ {1}(\om,\mathbb R^{3})$, and $ \v\cdot \n\in H^{3/2}(\partial\om)$. Then $\w\in H^ {2}(\om,\mathbb R^{3})$.
\end{lemma}

\begin{lemma}\lab{curl}Assume \eqref{OMEGA} and let $\v\in H^ {1}(\om,\mathbb R^{3})$ be such that $\mathrm{div}\, \v= 0$ a.e. in $\om$. Then there exists $\w\in H^ {2}(\om,\mathbb R^{3})$ such that  $\v=\curl\w,\ \dv\w=0$ a.e. in $\om$ and $\w\cdot\n=0$ on $\partial\om$.
\end{lemma}
\begin{proof} Since $\partial\om$ is connected then by  \cite[Lemma 3.5]{ABDG} there exists $\z\in H^ {1}(\om,\mathbb R^{3})$ such that  $\v= \curl \z,\ \dv \z=0$ a.e. in $\om$. By taking into account that
\begin{equation*}\displaystyle
\z\cdot \n\in H^{1/2}(\partial\om),\ \  \int_{\partial\om}\z\cdot \n\,d{\mathcal H}^{2}=\int_{\om} \dv \z\,dx=0,
\end{equation*}
there exists $\varphi\in H^{1}(\om)$ such that
\begin{equation*}
\left\{\begin{array}{ll}& \Delta\varphi=0\ \hbox{in}\ \om\\
&\\
&\displaystyle\frac{\partial\varphi}{\partial\n}=\z\cdot\n\ \hbox{in}\ \partial\om.\\
\end{array}\right.
\end{equation*}
By setting $\w:=\z-\nabla \varphi$ we get  $\w\cdot\n= 0$ on $\partial\om$, $\dv\w=0, \ \curl\w=\v\in H^ {1}(\om,\mathbb R^{3})$ a.e. in $\om$ and hence Lemma \ref{divrot} yields $\w\in H^ {2}(\om,\mathbb R^{3})$.
\end{proof}

It is well known that for a function $\boldsymbol\zeta\in H^1_0(\Omega,\mathbb R^3)$ there holds
\[
\int_\Omega |\nabla\boldsymbol\zeta|^2\,dx=\int_\Omega|\curl\boldsymbol\zeta|^2\,dx+\int_\Omega|\dv\boldsymbol\zeta|^2\,dx.
\] 
In presence of boundary values we have the following formula that we borrow from \cite{G}.

\begin{lemma}[$\mbox{\cite[Theorem 3.1.1.1]{G}}$] \lab{grisv} Assume \eqref{OMEGA} and let $\boldsymbol\zeta\in H^1(\om,\mathbb R^{3})$. Then 
\begin{equation*}\begin{aligned}\lab{grisvard} 
 \int_\om|\nabla\boldsymbol\zeta|^{2}\,dx&=\int_{\om}|\curl \boldsymbol\zeta|^{2}\,dx+\int_\Omega |\dv \boldsymbol\zeta|^{2}\,dx+ 2\left \langle\nabla(\boldsymbol\zeta\cdot \n)\wedge \n, \boldsymbol\zeta\wedge\n\right \rangle_{\partial\om}
\\
&\qquad-\int_{\partial\om}\left\{\dv \n\,(\boldsymbol\zeta\cdot\n)^2+(\boldsymbol\zeta\wedge\n)^T\nabla\n(\boldsymbol\zeta\wedge\n)\right\}\,d\mathcal H^2
\end{aligned}\end{equation*}
where $\langle\cdot,\cdot\rangle_{\partial\om}$ denotes the duality between $H^{-1/2}(\partial\om)$ and $H^{1/2}(\partial\om)$.
\end{lemma}
\begin{lemma}\lab{convforte} Assume \eqref{OMEGA} and \eqref{Gamma} 
 Let $(\boldsymbol\zeta_h)_{h\in\mathbb N}\subset H^1(\om,\mathbb R^3)$ be a sequence such that $\boldsymbol \zeta_h\cdot\n=0$ on $\Gamma$, $\boldsymbol \zeta_h\wedge\n=0$ on $\partial\om\setminus\Gamma$, $\boldsymbol \zeta_h\wconv 0$ weakly in $H^1(\om,\mathbb R^3)$, $\curl\ \boldsymbol\zeta_h\to 0$ in $L^2(\om,\mathbb R^3)$ and  $\dv\ \boldsymbol\zeta_h\to 0$ in $L^2(\om)$. Then $\boldsymbol\zeta_h\to 0$ strongly in $H^1(\om,\mathbb R^3)$.
 \end{lemma}
 \begin{proof} Since $\boldsymbol\zeta_h\cdot\n=0$ on $\Gamma$ and $\boldsymbol\zeta_h\wedge\n=0$ on $\partial\om\setminus\Gamma$ we get
\begin{equation*}
\left \langle\nabla(\boldsymbol\zeta\cdot \n)\wedge \n, \boldsymbol\zeta\wedge\n\right \rangle_{\partial\om}=0
\end{equation*}
for any $h\in\mathbb N$
and then Lemma \ref{grisv} yields 
\begin{equation*}\lab{grisvard2} 
\displaystyle \int_\om|\nabla\boldsymbol\zeta_h|^{2}\,dx= A_h+B_h,
\end{equation*}
where
\[
A_h:=\int_{\om}|\curl \boldsymbol\zeta_h|^{2}\,dx+\int_\Omega |\dv \boldsymbol\zeta_h|^{2}\,dx,
\]
\[
B_h:=-\int_{\partial\om}\left\{\dv \n\,(\boldsymbol\zeta_h\cdot\n)^2+(\boldsymbol\zeta_h\wedge\n)^T\nabla\n(\boldsymbol\zeta_h\wedge\n)\right\}\,d\mathcal H^2.
\]
By taking into account that $\boldsymbol \zeta_h\wconv 0$ weakly in $H^1(\om,\mathbb R^3)$, we get $\boldsymbol\zeta_h\to 0$ in $L^2(\partial\om,\mathbb R^3)$, hence $B_h\to 0$. Moreover,  since both
$\mathrm{curl}\ \boldsymbol\zeta_h$ and  $\mathrm{div}\ \zeta_h$ go to zero in $L^2(\om)$, we get also $A_h\to 0$. We conclude that $\|\nabla\boldsymbol\zeta_h\|_{L^2(\om)}\to 0$ and
the result follows  from the Poincar\'e inequality $\|\boldsymbol \zeta_h\|^2_{L^2(\om)}\le C\, (\|\nabla\boldsymbol \zeta_h\|^2_{L^2(\om)}+\|\boldsymbol\zeta_h\|^2_{L^2(\partial\om)})$.
\end{proof}


 Always assuming \eqref{OMEGA} and $\eqref{Gamma}$, 
 we set
\begin{equation*}
 X(\Gamma):=\{\boldsymbol \zeta\in H^1(\om,\mathbb R^3):\, \boldsymbol\zeta\cdot\n=0\ \hbox{ on}\  \Gamma, \ \boldsymbol\zeta\wedge\n=0\ \hbox{ on}\ \partial\om\setminus\Gamma\},
 \end{equation*}
\[
 X_0(\Gamma):=\{\boldsymbol\zeta\in X(\Gamma): \ \mathrm{div}\ \boldsymbol\zeta=0\ \hbox{in}\ \om\}.
\]
Lemma \ref{curl} ensures that there exist vector fields $\w\in H^2(\om,\mathbb R^3)$  such that  $ \mathrm{div}\, \w=0$ a.e. in $\om$,  $\w\cdot\n=0$ on $\partial\om$ and $\mathrm{curl} \w\cdot\n=0$ on $\Gamma$. 
Given $\w$ with such properties,
 for  $\boldsymbol\zeta\in X(\Gamma)$ we define the functional
\begin{equation*}\label{G}
\Phi_{\w}(\boldsymbol\zeta):=\frac{1}{2}\int_{\om}\left (|\hbox{curl} \boldsymbol\zeta|^{2}+|\hbox{div} \boldsymbol\zeta|^{2}\right )\,dx-\int_{\partial\om}(\w\wedge\boldsymbol\zeta)\cdot\n\,d\mathcal H^2.
\end{equation*}
We prove the following
\begin{lemma}\lab{minG} Assume \eqref{OMEGA} and \eqref{Gamma} 
Let $\w\in H^2(\om,\mathbb R^3)$ be such that  $ \mathrm{div}\, \w=0$ a.e. in $\om$,  $\w\cdot\n=0$ on $\partial\om$ and $\mathrm{curl} \w\cdot\n=0$ on $\Gamma$.  Then  the functional $\Phi_\w$ has  minimizers both on $X(\Gamma)$ and $X_0(\Gamma)$. Moreover, 
\begin{equation}\label{minmin}
\displaystyle\min_{X(\Gamma)} \Phi_\w= \min_{X_0(\Gamma)} \Phi_\w.
\end{equation}
\end{lemma}
\begin{proof} Let $(\boldsymbol\zeta_h)\subset X(\Gamma)$ be a sequence such that $\Phi_\w(\bb\zeta_h)\to \inf_{X(\Gamma)}\Phi_\w$. We shall first prove that such a sequence is bounded in $H^1(\om,\mathbb R^3)$. Indeed, assume by contradiction that, up to subsequences, $\lambda_h:=\|\boldsymbol\zeta_h\|_{H^1(\om,\mathbb R^3)}\to +\infty$ and set $\boldsymbol\xi_h:=\lambda_h^{-1}\boldsymbol\zeta_h$. Then $\boldsymbol\xi_h\in X(\Gamma), \|\boldsymbol\xi_h\|_{H^1(\om,\mathbb R^3)}=1$ and { for $h$ large enough}
\begin{equation*}
{1}\ge \Phi_\w(\boldsymbol\zeta_h)= \Phi_\w(\lambda_h\boldsymbol\xi_h)=\frac{\lambda_h^2}{2}\int_{\om}\left (|\hbox{curl}\ \boldsymbol\xi_h|^{2}+|\hbox{div} \boldsymbol\xi_h|^{2}\right )\,dx-\lambda_h\int_{\partial\om}(\w\wedge\boldsymbol\xi_h)\cdot\n\,d\mathcal H^2.
\end{equation*}
Hence, by recalling that $\|\boldsymbol\xi_h\|=1$, there exists $\boldsymbol \xi^*\in X(\Gamma)$ such that, up to subsequences, $\boldsymbol\xi_h\wconv\boldsymbol \xi^*$ weakly in $H^1(\om,\mathbb R^3)$ and
\begin{equation*}
\int_{\om}\left (|\curl \boldsymbol\xi_h|^{2}+|\dv \boldsymbol\xi_h|^{2}\right )\,dx\le \frac{2}{\lambda_h}\int_{\partial\om}(\w\wedge\boldsymbol\xi_h)\cdot\n\,d\mathcal H^2  + 2\lambda_h^{-2} \to 0,
\end{equation*}
that is, $\curl \boldsymbol\xi_h\to \hbox{curl}\ \boldsymbol\xi^*=0$ in $L^2(\om,\mathbb R^3)$, $  \hbox{div}\ \boldsymbol\xi_h\to \hbox{div}\ \boldsymbol\xi^*=0$ in $L^2(\om)$ and $\boldsymbol\xi^*\in X(\Gamma)$. We claim that $\boldsymbol \xi^*=0$.
Indeed let $\boldsymbol\psi\in H^1_0(\om,\mathbb R^3)$ be such that $\hbox{div}\ \boldsymbol\psi=0$. Since $\om$ is simply connected, then by  \cite[Theorem 3.17]{ABDG} there exists $\boldsymbol\omega\in H^2(\om,\mathbb R^3)$
such that $\hbox{curl}\ \boldsymbol\omega=\boldsymbol\psi$ a.e. in $\om$, $ \hbox{div}\ \boldsymbol\omega=0$ a.e. in $\om$ and $\boldsymbol\omega\wedge\n=0$ on $\partial\om$. Therefore
\begin{equation*}\begin{aligned}
\displaystyle\int_\om \boldsymbol\xi^*\cdot\boldsymbol\psi\,dx&=\int_\om\boldsymbol\xi^*\cdot\ \hbox{curl}\ \boldsymbol\omega\,dx=\int_\om\ \hbox{div}\ (\boldsymbol\xi^*\wedge\boldsymbol\omega)\,dx\\
&=\displaystyle \int_{\partial\om} (\boldsymbol\xi^*\wedge\boldsymbol\omega)\cdot\n\,d\mathcal H^2= \int_{\partial\om} \boldsymbol\xi^*\cdot (\boldsymbol\omega\wedge\n)\,d\mathcal H^2=0.
\end{aligned}
\end{equation*}
Hence, by invoking for instance \cite[Lemma 2.1]{GR}, there exists $p\in H^2(\om)$ such that $\nabla p=\boldsymbol\xi^*$ a.e. in $\om$, so that $\Delta p=0$ in $\om$, $\textstyle \frac{\partial p}{\partial\n}=0$ on $\Gamma$ and $\nabla p\wedge\n=0$ on $\partial\om\setminus\Gamma$. In particular, $p$ is constant on $\partial\om\setminus\Gamma$ since $\n\wedge(\nabla p\wedge\n)$ is its tangential derivative, hence $p$ is constant on the whole $\om$ and $\boldsymbol\xi^*=\nabla p=0$ as claimed.
By summarizing: $\boldsymbol\xi_h\in X(\Gamma)$, $\boldsymbol\xi_h\wconv 0$ weakly in $H^1(\om,\mathbb R^3)$, $\curl \boldsymbol\xi_h\to0$ in $L^2(\om;\mathbb R^3)$, $\dv \boldsymbol\xi_h\to 0$ in $L^2(\om)$. Therefore, by Lemma \ref{convforte}, $\boldsymbol\xi_h\to 0$ strongly in $H^1(\om,\mathbb R^3)$, which is  a contradiction since we are assuming $\|\bb\xi_h\|_{H^1(\om,\mathbb R^3)}=1$ (see also similar arguments in \cite{BBGT}, \cite{BT}, \cite{MPTARMA}, \cite{MPTFvK}).

Since we have shown that $(\boldsymbol\zeta_h)$ is a bounded sequence in $H^1(\om,\mathbb R^3)$ and
 since $\Phi_\w$ is sequentially l.s.c. with respect to the weak convergence in $H^1(\om,\mathbb R^3)$ we get existence of $\min_{X(\Gamma)}\Phi_\w$. It remains to prove that \eqref{minmin} holds.
Let $\boldsymbol\zeta^*\in \argmin_{X(\Gamma)} \Phi_\w$. Then $\hbox{div}\ \boldsymbol\zeta^*\in L^2(\om)$ and we let $\varphi\in H^2(\om)$ be  solution { (see for instance \cite[Theorem 3]{PL}) to the  boundary value problem}
\begin{equation*}\left\{\begin{array}{ll}
& \Delta \varphi= \ \hbox{div} \ \boldsymbol\zeta^* \ \hbox{in}\ \om\\
&\\
&\displaystyle \frac{\partial\varphi}{\partial\n}=0\ \hbox{on}\ \Gamma\\
&\\
& \varphi=0\ \hbox{on}\ \partial\om\setminus\Gamma.\\
\end{array}\right. 
\end{equation*}
It is readily seen that   $\boldsymbol\zeta^*-\nabla\varphi\in X_0(\Gamma)$ and 
\begin{equation*}
\Phi_\w(\boldsymbol\zeta^*-\nabla\varphi)= \Phi_\w(\boldsymbol\zeta^*)+\int_{\Gamma} (\w\wedge\nabla\varphi)\cdot \n\,d\mathcal H^2= \Phi_\w(\boldsymbol\zeta^*)+\int_{\partial\om} (\w\wedge\nabla\varphi)\cdot \n\,d\mathcal H^2
\end{equation*}
since, due  to $\varphi=0$ on $\partial\om\setminus\Gamma$, there holds
\begin{equation*}
\int_{\partial\om\setminus\Gamma} (\w\wedge\nabla\varphi)\cdot \n\,d\mathcal H^2=\int_{\partial\om\setminus\Gamma} \w\cdot (\nabla\varphi\wedge \n)\,d\mathcal H^2=0.
\end{equation*}
But 
\begin{equation*}\begin{aligned}
\displaystyle\int_{\partial\om} (\w\wedge\nabla\varphi)\cdot \n\,d\mathcal H^2&=\int_\om\hbox{ div}\ (\w\wedge\nabla\varphi)\,dx=-\int_\om \hbox{curl}\ \w\cdot\nabla\varphi\,dx\\
&\displaystyle=-\int_{\partial\om} (\hbox{curl}\ \w\cdot\n)\varphi\,d\mathcal H^2
=0
\end{aligned}
\end{equation*}
since $\mathrm{curl}\,\w\cdot\n=0$ on $\Gamma$ and $\varphi=0$ on $\partial\om\setminus\Gamma$.
We conclude that 
\begin{equation*}
\Phi_\w(\boldsymbol\zeta^*-\nabla\varphi)= \Phi_\w(\boldsymbol\zeta^*),
\end{equation*}
thus proving the result.
\end{proof}

The next result is based on the Euler-Lagrange equation for functional $\Phi_\w$. 
Before its statement, we recall that the {\it reach} of a closed set $A\subset\mathbb R^3$, 
  introduced in \cite{F}, is defined by
\begin{equation}\lab{reach}
R(A):=\sup\{r>0:  \ 0<d(x,A)< r\ \Rightarrow\;\; \exists ! \ y\in A \mbox{ s.t. } d(x,y)=d(x,A)\},
\end{equation}
where the {\it distance function} is defined on $\mathbb R^3$ by $d(x,A):=\inf_{y\in A}|x-y|$.
It  is well-known that $R(A)> 0$ whenever $A$ is a $C^{2}$ compact $1D$ or $2D$ manifold  without boundary, see for instance \cite{S,T}.

\begin{lemma}\lab{euler}  Assume \eqref{OMEGA} and \eqref{Gamma}. 
Let $\w$ as in {\rm Lemma \ref{minG}} and let $\boldsymbol\zeta_*\in \argmin_{X_0(\Gamma)}\Phi_\w$. Then we have 
\begin{equation}\lab{eulerG1}
\int_\om\curl\boldsymbol\zeta_*\cdot \mathrm{curl} \boldsymbol\varphi\,dx=0\ \quad\mbox{ $\forall\boldsymbol \varphi\in C^1_0(\Omega,\mathbb R^3)$}
\end{equation}
and 
\beeq\lab{eulerG2}
\displaystyle\langle \curl \boldsymbol\zeta_*\wedge\n, \boldsymbol \varphi\rangle_{\partial\om}=-\int_\Gamma (\w\wedge\n)\cdot\boldsymbol\varphi \, d\mathcal H^2 \quad\mbox { $\forall\boldsymbol\varphi \in C^{1}(\overline\om,\mathbb R^3)$ s.t. $\boldsymbol\varphi\equiv 0$ on $\partial\om\setminus\Gamma$. }
\end{equation}
\end{lemma}
\begin{proof} If $\boldsymbol \zeta_*\in \argmin_{X_0(\Gamma)}\Phi_\w$, then by Lemma \ref{minG} we have $\boldsymbol\zeta_*\in \argmin_{X(\Gamma)}\Phi_\w$.
If $\boldsymbol\varphi\in C^1(\overline\Omega,\mathbb R^3)$ is such that $\boldsymbol\varphi\wedge\n=0$ on $\partial\om\setminus\Gamma$ and $\boldsymbol\varphi\cdot\n=0$ on $\Gamma$,  for any $\varepsilon\in(-1,1)$ we have $\boldsymbol\zeta_*+\varepsilon\boldsymbol\varphi\in X(\Gamma)$. Hence,  the following first order condition holds
\begin{equation}\label{EL}
\int_\om\mathrm{curl}\boldsymbol\zeta_*\cdot\mathrm{curl}\boldsymbol\varphi\,dx=\int_{\partial\om}(\w\wedge\boldsymbol\varphi)\cdot \n\,d\mathcal H^2\quad\mbox{ $\forall\boldsymbol\varphi\in C^1(\overline\Omega,\mathbb R^3)\cap X(\Gamma)$. }
\end{equation}
Choosing in particular test functions $\boldsymbol\varphi$ that vanish on $\partial\om$, we deduce \eqref{eulerG1}.

 
 Let now $0<\delta< R(\partial\om)$, let  $\om_\delta:=\{x\in \om: d(x,\partial\om)< \delta\}$, so that for any $x\in\om$ there is  a unique projection $\sigma(x)$ of $x$ on $\partial\om$,  and set $\n(x):=\n(\sigma(x))$ for every $x\in \om_\delta$.
Let $\eta_\delta\in C^1(\overline\om)$ be a cutoff function such that  $\eta_\delta\equiv 1$ in $\om_{\delta/2}$ and  $ \eta_\delta\equiv 0$ in $\om\setminus\om_\delta$. Moreover,  for every $\bb\varphi \in C^{1}(\overline\om,\mathbb R^3)$ such that $\bb\varphi\equiv 0$ on $\partial\om\setminus\Gamma$, we define $\bb\varphi_\delta:=\eta_\delta\varphi$. We take advantage of the cutoff function to define $\bb\varphi_\delta\cdot \n$ on the whole $\Omega$ even if $\n$ is defined only on $\Omega_\delta$, as follows 
\begin{equation*}
(\varphi_\delta\cdot\n)(x):=\left\{\begin{array}{ll}
& \varphi_\delta(x)\cdot\n(x)
\ \hbox{if}\ x\in \om_\delta\\
&\\
& 0 \ \hbox{otherwise in}\ \om.
\end{array}\right.
\end{equation*}
It is readily seen that $\bb\theta_\delta:=
\bb\varphi_\delta- (\bb\varphi_\delta\cdot\n)\n\in X(\Gamma)\cap C^1(\overline\om,\mathbb R^3)$ and therefore we can make use of \eqref{EL} and get
\begin{equation}\label{nlab}
\int_\om \hbox{curl}\ \bb\zeta_*\cdot\ \hbox{curl}\ \bb\theta_\delta\,dx= -\int_\Gamma (\w\wedge\n)\cdot\bb\theta_\delta \, d\mathcal H^2=-\int_\Gamma (\w\wedge\n)\cdot\bb\varphi \, d\mathcal H^2
\end{equation}
A density argument shows that 
\begin{equation}\label{nnn}
\langle \hbox{curl}\ \bb\zeta_*\wedge\n, (\bb\varphi\cdot\n)\n\rangle_{\partial\om}=0.
\end{equation}
On the other hand, since \eqref{eulerG1} shows that $\curl\curl\bb\zeta_*=0$ a.e. in $\Omega$, integration by parts entails
\begin{equation}\label{320}
\int_\om \hbox{curl}\ \bb\zeta_*\cdot\ \hbox{curl}\ \bb\theta_\delta\,dx=\displaystyle\langle \hbox{curl}\ \bb\zeta_*\wedge\n, \bb\varphi-(\bb\varphi\cdot\n)\n\rangle_{\partial\om}.
\end{equation}
Combining  \eqref{nlab}, \eqref{nnn} and \eqref{320} yields the result. 
\end{proof}

For a divergence-free deformation field $\v$ that vanishes on $\Gamma$, taking advantage of the latter results we can construct a vector potential $\widetilde\w$ that vanishes on $\Gamma$ as well.

\begin{lemma}\lab{w=0} { Assume \eqref{OMEGA} and  \eqref{Gamma}}. 
 Let $\v\in H^{1}(\Omega,\R^3)$ such that $div\, \v=0$ a.e. in $\om$ and $\v=0$ on $\Gamma$. Then there exists $\widetilde\w\in H^{2}(\Omega,\R^3)$ such that $\widetilde\w=0$ on $\Gamma$ and $\hbox{curl}\, \widetilde\w=\v$  a.e. in $\om$.
 \end{lemma}
 \begin{proof} Let $\bb\zeta_*$ as in Lemma \ref{euler}. Then $\hbox{curlcurl}\ \bb\zeta_*=0$ a.e. in $\om$ and by arguing as in Lemma
 \ref{minG} with the help  \cite[Lemma 2.1]{GR} there exists $\theta\in H^1(\om)$ such that  $-\nabla\theta= \hbox{curl}\ \zeta_*$ a.e. in $\om$, hence by \eqref{eulerG2}  there holds $\nabla\theta\wedge\n=\w\wedge\n$ in the sense of $H^{-1/2}(\Gamma;\mathbb R^3)$. By taking into account that  $\w\in H^{3/2}(\partial\om,\mathbb R^3)$ and that $\n\wedge(\nabla\theta\wedge\n)$ is the (weak) tangential gradient of $\theta$ on $\Gamma$ we get $\theta\in H^{5/2}(\Gamma) $, so there exists $\widetilde\theta\in H^{5/2}(\partial\om) $ such that 
 $\widetilde\theta=\theta$ on $\Gamma$ { (thanks to the regularity of $\partial\Gamma$)}.
 Let now $\psi_*$ be  the unique solution  to the biharmonic boundary value problem
\begin{equation*}\left\{\begin{array}{ll}
&\Delta^{2}\psi=0\ \hbox{in}\ \om\\
&\\
&\psi=\widetilde\theta\ \hbox{on}\  \partial\om\\
&\\
&\displaystyle \frac{\partial\psi} {\partial\n}=0\ \hbox{on}\  \partial\om.\\
\end{array}\right.
\end{equation*}
Since $\widetilde\theta\in H^{5/2}(\partial\om) $ then $\psi_*\in H^{3}(\om)$ so by setting $\widetilde \w:=\w-\nabla \psi_* \in H^{2}(\om,\mathbb R^{3})$ we get $\widetilde \w\wedge\n=(\w-\nabla \psi_*)\wedge\n= \w\wedge\n-\nabla\theta\wedge\n=0$ on $\Gamma$ and $curl\, \widetilde\w=curl\, \w=\v$ a.e. in $\om$ and thesis follows by recalling that $\w\cdot\n=0$ on the whole $\partial\om$.
\end{proof}
{Lastly, we prove a property of traces of $H^2(\om;\mathbb R^3)$ functions which will be extremely useful in the proof of our main result.}
\begin{lemma}\lab{dernormwtang} Assume \eqref{OMEGA}  and \eqref{Gamma}. 
Let $\w\in H^{2}(\Omega,\R^3)$ such that $\w=0$ on $\Gamma$. Then
\begin{equation}\label{den}
\frac{\partial}{\partial\n}(\w\wedge\n) =0\  \ \hbox{ on}\ \Gamma
\end{equation}
if and only if
\begin{equation*}
\curl\w=0\  \ \hbox{on}\ \Gamma.
\end{equation*}
\end{lemma}
\begin{proof} 
As in the proof of Lemma \ref{euler} we define here $\n(x):=\n(\sigma(x))$ for every $x$ in a small neighborhood $\Omega_\delta$ of $\partial\om$, being $\sigma(x)$ the unique projection of $x$ on $\partial\om$. In particular, $\curl\n=0$ in $\overline{\Omega_\delta}$ so that $\partial_k\n=\nabla n_k$ on $\overline{\Omega_\delta}$ for any $k\in\{1,2,3\}$.  Suppose first that $\hbox{curl}\ \w=0$  on $\Gamma$. Then
we have on $\Gamma$ 
\begin{equation}\lab{dernorm}
\frac{\partial}{\partial\n}(\w\wedge\n)=\sum_{k=1}^3n_k\partial_k(\w\wedge\n)=\sum_{k=1}^3(n_k\partial_k\w\wedge\n+n_k\w\wedge\partial_k\n)=\sum_{k=1}^3 n_k\nabla\w_{k} \wedge\n
\end{equation}
since
$\sum_{k=1}^3 n_k\nabla n_k=0$. But $\nabla\w_{k} \wedge\n=0$ on $\Gamma$ for any $k\in\{1,2,3\}$ since $\mathbf w=0$ on $\Gamma$, implying that the tangential derivative of  $\w$  vanishes  on $\Gamma$.
Conversely if  \eqref{den} holds, by arguing as in \eqref{dernorm} with the help of $\nabla\w_{k} \wedge\n=0$  on $\Gamma$  we get
\begin{equation}\lab{curlgamma1}
\sum_{k=1}^3n_k(\partial_k\w-\nabla\w_k)\wedge\n=0\  \ \hbox{ on}\ \Gamma.
\end{equation}
Since by symmetry $\sum_{i=1}^3\sum_{k=1}^3n_in_k(\partial_k\w_i-\partial_i\w_k)=0$, from \eqref{curlgamma1} we get $$\sum_{k=1}^3n_k(\partial_k\w-\nabla\w_k)=0\ \hbox{ on}\ \Gamma.$$ Therefore, fixing  $k\in\{1,2,3\}$,
\begin{equation*}\begin{aligned}
\partial_k\w-\nabla\w_k&=\sum_{i=1}^3 (\partial_k\w_i-\partial_i\w_k)n_i\n+ \n\wedge((\partial_k\w-\nabla\w_k)\wedge\n)= \n\wedge((\partial_k\w-\nabla\w_k)\wedge\n) \\
&= \n\wedge(\partial_k\w\wedge\n)
=\n\wedge(\partial_k(\w\wedge\n)-\w\wedge\partial_k\n)=0
\end{aligned}
\end{equation*}
  on  $\Gamma$, where the latter equality is due to the fact that $\nabla(\w\wedge \n)=0$ $\hbox{ on}\ \Gamma$ (since  \eqref{den} holds and since  $\w\wedge \n$ vanishes on $\Gamma$). The result is proven.
\end{proof}

\section{Approximation results}\label{approximationsection}

Several approximation results will be needed for obtaining $\Gamma$-convergence and for giving the proof of the main theorem.
The first one is contained in the next lemma, which is a consequence of the well known {\it Reynolds' Transport Theorem}. We will prove it in some details since 
its application in this context seems a novelty, at least to our present knowledge.
\begin{lemma}\lab{reynolds} Let $\Omega$ be a bounded open subset of $\mathbb R^3$ and let $\Gamma\subseteq\partial\om$.  Let $\om'\subset \mathbb R^{3}$ be an open set such that $\overline\om\subset \om'$.  Let $\v\in C^1(\Omega ';\R^3)$ be such that $\mathrm{div}\, \v=0$  in $\om '$ and  $\v=0$ on $\Gamma$. Then, for every sequence $(h_j)_{j\in\mathbb N}$ of strictly positive numbers such that $\lim_{j\to\infty }h_{j}= 0$,    there exists a sequence $(\v_{j})_{j\in\mathbb N}\subset C^1(\om',\R^3)$ such that 
\begin{equation}\label{1}
\det (\mathbf I+h_{j}\nabla \v_{j})=1\ \hbox{  in}\ \om,
\end{equation}
\begin{equation}\label{2}
 \v_{j}=0\ \hbox{ on }\ \Gamma,
\end{equation} 
\begin{equation}\label{3}\v_{j}\to \v\ \hbox{ in}\  W^{1,p}(\om, \mathbb R^{3})\qquad \forall p\in [1,+\infty),\end{equation}
\begin{equation}\label{4} \|h_{j}\nabla \v_{j}\|_{L^{\infty}(\om,\mathbb R^{3\times 3})}\to 0.\end{equation}

 \end{lemma}

\begin{proof} 
Let $\om_*$ be an open set, compactly contained in $\Omega'$, such that $\overline\om\subset\Omega_*$. We choose $T\in(0,1)$ small enough, such that $\mathbf y(t,x)\in\om'$ for any $x\in\om_*$ and any $t\in[0,T]$, where $\mathbf y(\cdot,x)$ is the unique solution to
\beeq \lab{flow}\left\{\begin{array}{ll} &\displaystyle\frac{\partial  {\bf y}}{\partial t}(t,x)=\v({\bf y}(t,x)),\qquad t\in(0,T]\\
&\\
& {\bf y}(0,x)=x.
\end{array}\right.
\eneq
We are indeed in the setting of \cite[Corollary 5.2.8, Remark 5.2.9]{HP}: we have that $\mathbf y\in C^1([0,T]; W^{1,\infty}(\om_*))$
is the flow associated to the vector field $\v$.
Moreover (letting $\nabla$ denote the derivative in the $x$ variable) $\mathbf Z(t,x):=\nabla \mathbf y(t,x)$  satisfies
\beeq \lab{flow2}\left\{\begin{array}{ll} &\displaystyle\frac{\partial  { \mathbf Z}}{\partial t}(t,x)=\nabla \v({\bf y}(t,x))\mathbf Z(t,x)\qquad t\in(0,T],\\
&\\
& {\mathbf Z}(0,x)=\mathbf I,
\end{array}\right.
\eneq
for a.e. $x\in\om_*$, { hence $\mathbf Z\in C^1([0,T]; L^\infty(\om_*))$ and
\begin{equation*}
\mathbf Z(t,x)=\exp\left (\int_0^t\nabla\v(\mathbf y(s,x))\,ds\right )
\end{equation*}
 for every $t\in [0,T]$ and for a.e. $x\in\om_*$. Therefore,
 \begin{equation}\lab{det}
\det \mathbf Z(t,x)=\exp\left (\int_0^t\tr\,\nabla\v(\mathbf y(s,x))\,ds\right )=\exp\left (\int_0^t\dv\v(\mathbf y(s,x))\,ds\right )=1
\end{equation}
for every $t\in [0,T]$ and for a.e. $x\in\om_*$.
 
%

Assuming wlog that $h_j<T$, we define $${\bf y}_{j}(x):={\bf y}(h_{j},x),\qquad  \v_{j}(x):= h_{j}^{-1}({\bf y}_{j}(x)-x),\qquad x\in\om_*.$$
 By taking into account that $\v=0$ on $\Gamma\subset \om_*$ we get ${\mathbf y}_{j}(x)\equiv x$ on $\Gamma$ so $\v_j$ vanishes on $\Gamma$ and \eqref{det} entails   $\det (\mathbf I+h_{j}\nabla \v_{j})=1$ a.e. in $\om$, thus proving \eqref{1} and \eqref{2}. }

We next prove \eqref{3}. Let $t\in [0,T]$. We notice that from \eqref{flow} we get
\begin{equation}\label{diff}
\frac1t\left(\mathbf y(t,x)-x\right)-\mathbf v(x)=\frac1t\int_0^t(\mathbf v(\mathbf y(s,x))-\mathbf v(x))\,ds
\end{equation}
and thus
\[
\frac1t|\mathbf y(t,x)-x|\le \|\mathbf v\|_{W^{1,\infty}(\Omega')}\int_0^t\frac1s|\mathbf y(s,x)-x|\,ds
\]
for any $x\in\Omega$,
and Gronwall lemma entails 
\[
\frac1t|\mathbf y(t,x)-x|\le |\mathbf v(x)| \exp\{\|\mathbf v\|_{W^{1,\infty}(\Omega')}t\}\le C_{\mathbf v},
\]
where $C_{\mathbf v}:=  {\|\mathbf v\|_{W^{1,\infty}(\Omega')}} \exp\{\|\mathbf v\|_{W^{1,\infty}(\Omega')}\}$. 
From the definition of $\mathbf v_j$, from \eqref{diff} and from the latter estimate we obtain
\[\begin{aligned}
\displaystyle\left |\v_{j}(x)- \v(x)\right |&=\left|\frac1{h_j}(\mathbf y(h_j,x)-x)-\mathbf v(x)\right|	\le  \|  \v\|_{W^{1,\infty}(\Omega)}\, \int_{0}^{h_{j}}\frac1s\,\left |{\bf y}(s,x)-x\right |\,ds\\&\le C_{\v}\|\v\|_{W^{1,\infty}(\Omega)}\,h_j
\end{aligned}\]
for any $x\in\Omega$ and any $j\in\mathbb N$. From the latter we get the convergence of $\v_j$ to $v$ in $L^1\cap L^\infty(\Omega)$ as $j\to0$.

We take now the gradient in \eqref{diff}, and since the map $\Omega_*\ni x\mapsto\v(\mathbf y(t,x))$ is Lipschitz continuous for any $t\in(0,T]$ we may take the gradient under integral sign and obtain
\begin{equation}\label{grad}
\begin{aligned}
&\frac1t (\nabla \mathbf y(t,x)-\mathbf I)-\nabla\mathbf v(x)=\frac1t\int_0^t\left(\nabla[\mathbf v(\mathbf y(s,x))]-\nabla \mathbf v(x)\right)\,ds\\
&=\frac1t \int_0^t\left(\nabla\mathbf v(\mathbf y(s,x))\nabla \mathbf y(s,x)-\nabla \mathbf v(x)\nabla\mathbf y(s,x)\right)\,ds+
\frac1t\int_0^t\left(\nabla \mathbf v(x)\nabla\mathbf y(s,x) -\nabla \mathbf v(x)\right)\,ds
\end{aligned}
\end{equation}
for a.e. $x\in\Omega$. Form the first equality of \eqref{grad} we get  
\begin{equation}\label{ad4}\begin{aligned}
\frac1t \|\nabla \mathbf y(t,\cdot)-\mathbf I\|_{L^\infty(\om_*)}&=\frac 1t\left\|\int_0^t\nabla \v(\mathbf y(s,\cdot))\nabla\mathbf y(s,\cdot)\,ds\,\right\|_{L^\infty(\om')}\\&\le \|\v\|_{W^{1,\infty}(\om_*)} \sup_{t\in[0,T]}\|\nabla \mathbf y(t,\cdot)\|_{L^\infty(\om_*)}=\mathcal Q \|\v\|_{W^{1,\infty}(\om_*)}
\end{aligned}\end{equation}
for any $t\in(0,T]$, where $\mathcal Q:=\|\nabla \mathbf y\|_{C^0([0,1]; L^\infty(\om_*))}<+\infty$ recalling \eqref{flow2}.
Still from the first equality of \eqref{grad} we have
\begin{equation*}
\begin{aligned}
&|\nabla\v_j(x)-\nabla \v(x)|=\left|\frac1{h_j}(\nabla\mathbf y(h_j,x)-I)-\nabla \v(x)\right|\\&\le \int_0^{h_j}\frac{|\nabla \mathbf y(s,x)|}{h_j}\,|\nabla \v(\mathbf y(s,x))-\nabla \v(x)|\,ds+\| \v\|_{W^{1,\infty}(\om')}\int_0^{h_j}\frac{|\nabla \mathbf y(s,x)-I|}{s}\,ds, 
\end{aligned}
\end{equation*}
for a.e. $x\in\om$ and any $j\in\mathbb N$.
Hence, \eqref{ad4} entails
\begin{equation}\label{tointegrate}
|\nabla\v_j(x)-\nabla \v(x)|\le \frac {\mathcal Q}{h_j}\int_0^{h_j}|\nabla \v (\mathbf y(s,x))-\nabla \v(x)|\,ds+\mathcal Q \|\v\|_{W^{1,\infty}(\om')} h_j
\end{equation}
for a.e. $x\in\om$ and any $j\in\mathbb N$.
We notice that for $p\in[1,+\infty)$, by Jensen inequality there holds
\[\begin{aligned}
\frac1{h_j^p}\left(\int_0^{h_j} |\nabla \v(\mathbf y(s,x))-\nabla \v(x)|\,ds\right)^p&\le\frac1{h_j}\int_0^{h_j}|\nabla\v(\mathbf y(s,x))-\nabla \v(x)|^p\,ds
\\&\le2\|v\|_{W^{1,\infty}(\om')}^{p-1}\frac1{h_j}\int_0^{h_j}|\nabla \v(\mathbf y(s,x))-\nabla \v(x)|\,ds
\end{aligned}\]
and the above right hand side vanishes for any $x\in\om$   as $j\to+\infty$ since $\nabla \v$ is continuous by assumption, so that we obtain
\begin{equation}\label{dominated}
\lim_{j\to\infty} \frac1{h_j^p}\int_\om\left(\int_0^{h_j} |\nabla \v(\mathbf y(s,x))-\nabla \v(x)|\,ds\right)^p\,dx=0
\end{equation}
by dominated convergence, using $2\|v\|_{W^{1,\infty}(\om')}^{p}$ as dominating function on the bounded domain $\om$.
From \eqref{tointegrate} we find
\[
\int_\om |\nabla\v_j(x)-\nabla \v(x)|^p\,dx\le \frac{\mathcal Q^p}{h_j^p}\int_\om\left(\int_0^{h_j}|\nabla \v(\mathbf y(s,x))-\nabla \v(x)|\,ds\right)^p\!\!\!dx+|\Omega|\mathcal Q^p\|\v\|_{W^{1,\infty}(\om_*)}^ph_j^p
\]
so that the $L^p(\Omega)$ convergence of $\nabla \v_j$ to $\nabla \v$ follows by taking the limit as $j\to+\infty$ and by using \eqref{dominated}. This concludes the proof of \eqref{3}.

Eventually, since $\nabla \v_j(x)=\frac1{h_j}(\nabla\mathbf y(h_j,x)-\mathbf I)$, \eqref{4} directly follows from \eqref{ad4}.
\end{proof}


The next step is an approximation of divergence-free $H^1(\Omega,\mathbb R^3)$ vector fields with divergence-free $C^1(\overline\om,\mathbb R^3)$ vector fields, in presence of suitable vanishing conditions on subsets of $\partial\om$. It is stated in Lemma \ref{gammal}. It requires the introduction of some notation about normal bundles and a couple of preliminary lemmas.
From here and through the rest of the paper, assumptions \eqref{OMEGA} and \eqref{Gamma} are always understood to hold.

Recalling the definition of reach from \eqref{reach}, with the convention $R(\emptyset)=+\infty$, let
\begin{equation}\label{mu0}
\mu_0:=\frac12\,\min\{R(\partial\om),R(\partial\Gamma)\}.
\end{equation}

\begin{remark}[Regularity of the squared distance function]\label{distancefunction}\rm
Assume \eqref{OMEGA} and \eqref{Gamma}. Let either $A=\partial\Gamma\neq\emptyset$ or $A=\partial\om$.
The distance function $d(x,A):=\min_{y\in A}|x-y|$ is differentiable at any point $x\in\mathbb R^3$ such that $0<d(x,A)<R(A)$, see \cite[Theorem 3.3, Chapter 6]{DZbook}. In particular, the squared distance function $d^2(\cdot, A)$ inherits the $C^3$ regularity of $A$ in the tubular neighbor $U_0(A):=\{x\in\mathbb R^3: d(x,A)<\mu_0\}$, see for instance \cite[Proposition 4.6]{MM}, see also \cite[Theorem 6.5, Chapter 6]{DZbook}. 
We deduce $d(\cdot, A)\in C^3(U_0(A)\setminus A)$.

\end{remark}

 For every $0< \mu< \mu_0$, let
 \begin{equation}\begin{array}{ll}\lab{SpartS}
 & S_{\mu}:=\{\sigma+t\n(\sigma):\ \sigma\in \Gamma,\ |t|\le \mu\}.
 \end{array}\end{equation}
We further define, for any
$0< \mu< \mu_0$ and any $0<\delta<\mu_0$,
%
 \begin{equation}\lab{defgamma}\begin{aligned}
  & \Gamma_{\delta}:=\{x\in \partial\om: d(x, \Gamma)\le \delta\},\qquad S_{\mu,\delta}:=\{\sigma+t\n(\sigma):\ \sigma\in  \Gamma_{\delta}\}.\end{aligned}
 \end{equation}
In case  $\Gamma=\partial \om$ (i.e., $\partial\Gamma=\emptyset$), we have $\Gamma_\delta\equiv\Gamma$ and $S_{\mu,\delta}\equiv S_\mu$, for any $0<\delta<\mu_0$. We stress that this case in encoded in Lemma \ref{eps4} and Lemma \ref{gammal} below. { On the other hand, if $\partial\Gamma\neq\emptyset$, then assumption \eqref{OMEGA} and the regularity properties of the distance function from $\partial\Gamma$  (see Remark \ref{distancefunction}) imply that  $\partial\Gamma_\delta$ is a $C^3$ one-dimensional submanifold of $\partial\Omega$. In particular, $\Gamma_\delta$ itself satisfies assumption \eqref{Gamma}}.

  Let us also introduce the following notation for  neighbors of $\om$ and $\partial\om$
 \begin{equation}\label{om''}
 \Omega':=\{x\in \mathbb R^3: d(x, \om)<\mu_0\}\quad\mbox{and}\quad\Omega'':=\{x\in\mathbb R^3: d(x,\partial\om)<\mu_0\}\subset\om'.
 \end{equation}
We notice that  $\mathbf n$ can be extended to $\Omega''$ in the usual way: $\mathbf n(x)=\mathbf n(\sigma(x))$, where $\sigma(x)$ is the unique projection of $x\in\om''$ on $\partial\om$, and therefore  $\mathbf n$ is a  $C^{2}(\Omega'',\mathbb R^3)$ vector field, so that there exists $K>1$ such that \begin{equation}\label{KAPPA}|\nabla\mathbf n|+|\nabla^2\mathbf n|\le K\;\;  \mbox{in $\Omega''$}.\end{equation} 

Some auxiliary estimates are given by the next
\begin{lemma}\label{eps4} Assume 
\eqref{OMEGA} and
\eqref{Gamma}. Let 
 $0<\delta<\mu_0$, where $\mu_0$ is defined by \eqref{mu0}.
Let $\bb f\in H^2(\om'',\mathbb R^3)$ be such that $\bb f=0$ on $\Gamma_{\delta}$.
Then there exists $\eps_0\in(0,\mu_0)$ such that for any $\eps\in(0,\eps_0)$ and any $\lambda\in(0,\delta)$ there holds
\begin{equation}\label{chart1}
\int_{S_{\eps,\lambda}} |\bb f|^2\,dx\le
2\eps^2 \int_{S_{\eps,\lambda}}|\nabla\bb f|^2\,dx,
\end{equation}
and if $\ \dfrac{\partial\bb f}{\partial\mathbf n}=0$  on $\Gamma_{\delta}$ as well,  there holds
\begin{equation}\label{chart2}
\int_{S_{\eps,\lambda}}|\bb f|^2\,dx
\le \frac{\eps^4}2 \int_{S_{\eps,\lambda}}|\nabla^2\bb f|^2\,dx.
\end{equation}
\end{lemma} 
\begin{proof}
Let $B$ denote the unit ball in $\mathbb R^2$ and 
let  $\bb\psi\in C^{3}(\overline B; \mathbb R^3) $ be any local chart parametrizing
a subset of $\partial\om$. Let $B_{\lambda}:=\bb\psi^{-1}(\Gamma_{\lambda}\cap\bb\psi(B))$. Let $ B\times(-\eps,\eps)\ni (u,t)\mapsto\bb\Phi(u,t):=\bb\psi(u)+t\mathbf n(\bb\psi(u))$.  Up to covering  $\Gamma_{\lambda}$ with local charts, it is enough to show that \eqref{chart1} and \eqref{chart2} hold, for suitably small $\eps$, with $\bb\Phi(B_{\lambda}\times(-\eps,\eps))$ in place of $S_{\eps,\lambda}$.
 We have $|\det D \bb\Phi(u,0)|=|\partial_1\bb\psi(u)\wedge\partial_2\bb\psi(u)|>0$, where $D$ denotes the gradient in the variables $(u,t)$.   $|\det D \bb\Phi(u,0)|$ is bounded away from zero on $B$ 
and  $|\det D \bb\Phi(u,t)|=|\partial_1\bb\psi(u)\wedge\partial_2\bb\psi(u)|+o(1)$ as $t\to 0$, uniformly with respect to $u\in B$. 
We notice that
 $|D\bb\Phi|$ is bounded   on $B\times(-\eps,\eps)$. Moreover, it is not difficult to check that for any small enough $\eps$ there holds
\begin{equation}\label{detratio}
\frac{|\det D\bb\Phi(u,t)|}{|\det D \bb\Phi(u,s)|}\le 2\quad\mbox{for any $(u,t,s)\in B\times(-\eps,\eps)\times(-\eps,\eps)$}.
\end{equation}
 
By the properties of Sobolev functions (see for instance \cite[Chapter1]{M}),
$\bb f\circ \bb \Phi\in H^2(B\times(-\eps,\eps))$, as $\bb\Phi$ is a $C^{2}$ homemorphism whose Jacobian is bounded away from $0$ and $+\infty$ on $B\times(-\eps,\eps)$, 
and $t\mapsto \bb f(\Phi(u,t))$ is absolutely continuous for $\mathcal L^2$-a.e. $u\in B$. Thus
  for $\mathcal L^2$-a.e $u\in B$ we get
\[
\bb f(\bb\Phi(u,t))=\bb f(\bb \Phi(u,0))+t\int_0^1\frac{d(\bb f\circ\bb\Phi)}{ds}\left(u,st\right)\,ds,
\]
so that by Jensen inequality, and since $\bb f=0$ on $\Gamma$ and $|\partial_t\bb\Phi|=1$, we obtain
\[
|\bb f(\bb\Phi(u,t))|^2\le |t|\int_{0\wedge t}^{0\vee t}|\nabla\bb f(\bb\Phi(u,s))|^2\,ds.
\]
By the latter inequality, by changing variables and by \eqref{detratio} we get for any small enough $\eps$
\[\begin{aligned}
&\int_{ \Phi(B_{\lambda}\times(-\eps,\eps))}\left|\bb f\right|^2\,dx=\int_{B_{\lambda}}\int_{-\eps}^{\eps}
|\bb f(\bb\Phi(u,t))|^2\,|\det D \bb\Phi(u,t)|\,dt\,du \\
&\qquad\le \int_{B_{\lambda}} \int_{-\eps}^{\eps} |t| \left(\int_{0\wedge t}^{0\vee t}|\nabla \bb f(\bb\Phi(u,s))|^2\,ds\right)|\det D\bb\Phi(u,t)|\,dt\,du\\
&\qquad\le \int_{B_{\lambda}} \int_{-\eps}^{\eps} 2|t| \int_{0\wedge t}^{0\vee t}|\nabla \bb f(\bb\Phi(u,s))|^2\,|\det D\bb\Phi(u,s)|\,{ds}\,dt\,du\\
&\qquad\le \int_{B_{\lambda}} \int_{-\eps}^{\eps} 2|t|\,dt \int_{-\eps}^{\eps}|\nabla \bb f(\bb\Phi(u,s))|^2\,|\det D\bb\Phi(u,s)|\,{ds}\,du\\
&\qquad\le 2\eps^2\int_{B_{\lambda}}\int_{-\eps}^{\eps}|\nabla\bb f(\bb\Phi(u,\tau))|^2\,|\det D\bb\Phi(u,s)|\,ds\,du=2\eps^2\iint_{\bb\Phi(B_{\lambda}\times(-\eps,\eps))}|\nabla \bb f|^2\,dx.
\end{aligned}\]


Similarly,   under the further null trace assumption of $\frac{\partial\bb f}{\partial\n}$ on $\Gamma_{\mu,\delta}$ we deduce that
\[
\frac{d}{dt}\left(\bb f(\bb\Phi(u,t))\right){\Big{|}_{t=0}}=\frac{\partial\bb f}{\partial \bb n}(\bb\psi(u))
\]
vanishes  as well  and we obtain for $\mathcal L^2$-a.e $u\in B$, since $\partial^2_t\bb\Phi=0$,
\[\begin{aligned}
|\bb f(\bb\Phi(u,t))|&=\left|\bb f (\bb\Phi(u,0))+t\,\frac{d}{dt}\left(\bb f(\bb\Phi(u,t))\right){\Big{|}_{t=0}}
+\frac12 t^2\int_0^1 \frac{d^2(\bb f\circ\bb\Phi)}{ds^2}\left(u,st\right)\,ds\right|\\
&\le \frac12 t^2\int_0^1|\nabla^2\bb f(\bb\Phi(u,st))|\,ds,
\end{aligned}\]
thus
\[
|\bb f(\bb\Phi(u,t))|^2\le \frac{|t|^3}{2}\int_{0\wedge t}^{0\vee t}|\nabla^2\bb f(\bb\Phi(u,s))|^2\,ds.
\]
Arguing as above we get for any small enough $\eps$
\[\begin{aligned}
&\int_{ \Phi(B_{\lambda}\times(-\eps,\eps))}\left|\bb f\right|^2\,dx=\int_{B_{\lambda}}\int_{-\eps}^{\eps}
|\bb f(\bb\Phi(u,t))|^2\,|\det D \bb\Phi(u,t)|\,dt\,du \\
&\quad\le
\int_{B_{\lambda}}\int_{-\eps}^{\eps} \frac{|t|^3}2\left(\int_{0\wedge t}^{0\vee t}|\nabla^2\bb f(\Phi(u,s))|^2\,ds\right)|\det D \bb\Phi(u,t)|\,dt\,du\\
&\quad\le 
\int_{B_{\lambda}}\int_{-\eps}^{\eps} {|t|^3}\,dt\int_{-\eps}^\eps|\nabla^2\bb f(\Phi(u,s))|^2\,|\det D \bb\Phi(u,s)|\,ds\,du
=\frac{\eps^4}2 \int_{ \Phi(B_{\lambda}\times(-\eps,\eps))}\left|\nabla^2\bb f\right|^2\,dx,
\end{aligned}\] 
as desired.
\end{proof}
\KKK

We are ready for the statement of the approximation result
{\begin{lemma}\lab{gammal}
 Assume 
\eqref{OMEGA} and
\eqref{Gamma}.
 Let  $0<\delta<\mu_0$, where $\mu_0$ is defined by \eqref{mu0}.
 Let $ \v\in H^{1}(\om,\mathbb R^{3})$ such that $div\, \v=0$ a.e. in $\om$ and $\v=0$ on $\Gamma_{\delta}$. Then there exists
  a sequence
 $(\v_{j})_{j\in\mathbb N}\subset C^1(\overline{\om '},\R^3)$ such that $\mathrm{div}\, \v_{j}=0$ a.e. in $\om '$, $\v_{j}=0$ on $\Gamma$ and $\v_{j}\to \v$ in   $H^{1}(\om,\mathbb R^{3})$.
 \end{lemma}
 \begin{proof} 
Let $0<\mu<\mu_0$. 
Recalling  \eqref{SpartS}, \eqref{TSigma} and  \eqref{defgamma}, 
notice that if $\lambda\in(0,\delta/2)$,  $\eps\in(0,\mu/2)$  then there hold
$S_{2\eps,2\lambda}\subset S_{\mu,\delta}$ and
$ S_{2\eps,2\lambda}\cap\partial\om\subset\Gamma_{\delta}$, which will be crucial for the proof: the projection on $\partial \om$ of any point in $S_{2\eps,2\lambda}$ lies on $\Gamma_{\delta}$. On the other hand it  clear that $S_{2\eps,2\lambda}\subset\om''$, where $\om''$ is defined by \eqref{om''}, thus $\n$ is well defined on $S_{2\eps,2\lambda}$.

   Let $\zeta\in C^{2}(\mathbb R)$ be defined by
  \begin{equation*}	\label{poly}
\zeta(\xi):=\left\{\begin{array}{ll}
 (3\xi^2-2\xi^3)^2\quad &\hbox{if}\ \xi\in [0,1]\\
0\quad&\hbox{if}\ \xi<0\\
 1\quad &\hbox{if}\ \xi>1.\\
\end{array}\right.
\end{equation*}
We introduce some notation: $ S^{*}_{\eps,\lambda}:=S_{2\eps,\lambda}\setminus S_{\eps,\lambda},\  S^{**}_{\eps,\lambda}:=S_{2\eps,2\lambda}\setminus S_{2\eps,\lambda}$, $\tilde S_{\eps,\lambda}:=S^{*}_{\eps,\lambda}\cup S^{**}_{\eps,\lambda}$.
Let $\lambda\in(0,\delta/2)$, let  $\eps\in(0,\lambda\wedge(\mu/2))$ and let 
\begin{equation*}
\eta_{\eps,\lambda}(x):=\!\!\left\{ \begin{array}{ll}&\!\!\!\! 0\ \hbox{if}\ x\in S_{\eps,\lambda}\\
&\\
&\displaystyle\!\!\!\!\!\zeta \!\left (\frac{d(x,\partial\om)-\eps}{\eps}\right )\ \hbox{if}\ x\in S^*_{\eps,\lambda}\\
&\\
&\displaystyle\!\!\!\!\!\zeta \!\left (\frac{d(x,\partial\om)-\eps}{\eps}\right )\!+\!\zeta\!\left (\frac{d(\sigma(x),\partial\Gamma)-\lambda}{\lambda}\! \right )\!\!\left(1-\zeta \!\left (\frac{d(x,\partial\om)-\eps}{\eps}\!\right )\!\!\right)\ \hbox{if}\ x\in S^{**}_{\eps,\lambda}\\
&\\
&\!\!\!\! 1\ \hbox{otherwise in}\ \mathbb R^3\\
\end{array}\right.
\end{equation*}
We notice that in the particular case $\partial\Gamma=\emptyset$, we have $S_{\eps,\lambda}=S_\eps$, $S_{\eps,\lambda}^*=S_{2\eps}\setminus S_\eps$, $S_{\eps,\lambda}^{**}=\emptyset$, and in fact $\eta_{\eps,\lambda}$ does not depend on $\lambda$.  
Taking advantage of the $C^{3}$ regularity of $d(\cdot, \partial\om)$ in $\om''\cap\{x\in\mathbb R^3:d(x,\partial\om)>\eps/2\}$, of the $C^2$ regularity of $\sigma$ in $\om''$ and of the $C^{2}$ regularity of $x\mapsto d(\sigma(x),\partial\Gamma)$ in $S_{2\eps,2\lambda}\setminus S_{2\eps,\lambda}$ (see Remark \ref{distancefunction}),  it can be easily checked that $\eta_{\eps,\lambda}\in C^{1,1}(\mathbb R^3)$. Moreover,  we have $|\nabla d(x,\partial\om)|\le 1$ and $|\nabla^2 d(x,\partial\om)|\le C_*/\eps$ on $S^*_{\eps,\lambda}$ (and similarly, $|\nabla(d(\sigma(x),\partial\Gamma))|\le C_*$,  and $|\nabla^2(d(\sigma(x),\partial\Gamma))|\le C_*/\lambda$ on $S^{**}_{\eps,\lambda}$) for some $C_*>$ that is independent of $\eps$ and $\lambda$. Taking advantage of these distance estimates, a computation shows  that there is a constant $C>0$ (not depending on $\eps$ and $\lambda$) such that 
 \begin{equation}\lab{eta}
 \begin{aligned}
&   \nabla \eta_{\eps,\lambda}\cdot\n=0\ \hbox{in} \ S_{\eps,\lambda},\qquad  2\,|\nabla \eta_{\eps,\lambda}\cdot\n|\le \dfrac C{\eps}\ \hbox{in}\ \tilde S_{\eps,\lambda}\\
& \n\wedge(\nabla \eta_{\eps,\lambda}\wedge\n)=0\ \hbox{in} \ S^{*}_{\eps,\lambda},\qquad 2\,|\n\wedge(\nabla \eta_{\eps,\lambda}\wedge\n)|\le \dfrac C{\lambda}<\dfrac C{\eps}\ \hbox{in} \ S^{**}_{\eps,\lambda},\\
&2|\nabla(\nabla\eta_{\eps,\lambda}\cdot\mathbf n)|\le \dfrac{C}{\eps^2}\ \hbox{in}\ S^{*}_{\eps,\lambda},\qquad 2|\nabla(\n\wedge(\nabla\eta_{\eps,\lambda}\wedge\mathbf n))|\le \dfrac{C}{\lambda\eps}\ \hbox{in} \ S^{**}_{\eps,\lambda}.
\end{aligned}
\end{equation}

Thanks to Lemma \ref{w=0}, we find  $\widetilde\w\in H^{2}(\Omega,\R^3)$ such that $\widetilde\w=0$ on $\Gamma_{\delta}$ and $\hbox{curl}\, \widetilde\w=\v$ a.e. in $\om$ (in particular, $\mathrm{curl}\,\widetilde \w=0$ on $\Gamma_{\delta}$). Let us consider  a $H^2(\mathbb R^3, \mathbb R^3)$ extension of $\widetilde\w$, still denoted by $\widetilde\w$, and set $$\w_{\eps,\lambda}:=\eta_{\eps,\lambda}\widetilde\w.$$ Therefore, $\w_{\eps,\lambda}\in H^2(\om',\mathbb R^3),\ \w_{\eps,\lambda}= 0$ on $\Gamma_{\delta}$  and
\begin{equation}\lab{curlweps}
\curl\ \w_{\eps,\lambda}= \eta_{\eps,\lambda}\curl{\widetilde\w}\ -\widetilde\w\wedge\nabla \eta_{\eps,\lambda}
\end{equation}
so that $\curl\ \w_{\eps,\lambda}=0$ on $\Gamma_{\delta}$ as well.
 Moreover, for $i=1,2,3$,
\begin{equation}\lab{gradcurlweps}\begin{array}{ll}
&\partial_i\hbox{curl}\ \w_{\eps,\lambda}= \partial_i\eta_{\eps,\lambda}\curl{\widetilde\w}\ +\eta_{\eps,\lambda}\partial_i\curl{\widetilde\w}-\partial_i\widetilde\w\wedge\nabla \eta_{\eps,\lambda}-\widetilde\w\wedge\partial_i\nabla \eta_{\eps,\lambda}\\
\end{array}
\end{equation}
and it is readily seen that, as $\eps\to 0^+$, 
\begin{equation}\label{twoconvergence}\eta_{\eps,\lambda}\curl{\widetilde\w}\to \curl{\widetilde\w}\quad\mbox{ and }\quad \eta_{\eps,\lambda}\partial_i\curl{\widetilde\w}\to
\partial_i\curl{\widetilde\w}
\end{equation}
 in $L^2(\om,\mathbb R^3)$. 
 We  claim that 
\begin{equation}\label{fourconvergence}
\widetilde\w\wedge\nabla \eta_{\eps,\lambda}\to 0,\qquad \partial_i\eta_{\eps,\lambda}\curl{\widetilde\w}\to0,\qquad\partial_i\widetilde\w\wedge\nabla \eta_{\eps,\lambda}\to0,\qquad\widetilde\w\wedge\partial_i\nabla \eta_{\eps,\lambda}\to 0
\end{equation}
in $L^2(\om,\mathbb R^3)$ as $\eps\to 0^+$, which will then imply, along with \eqref{curlweps}-\eqref{gradcurlweps}-\eqref{twoconvergence}, that
\begin{equation}\lab{curlwtov}
\curl\w_{\eps,\lambda}\to \curl\widetilde\w=\v\ \ \hbox{in}\ H^1(\om,\mathbb R^3)\ \hbox{as}\ \eps\to 0^+.
\end{equation}
In order to prove the claim, we separately treat the four terms in \eqref{fourconvergence}, by taking into account \eqref{eta} and  
 Lemma \ref{dernormwtang}, which entails
$\widetilde\w=\widetilde\w\wedge\mathbf n=\frac{\partial}{\partial\n}(\widetilde\w\wedge\n)=0$ on $\Gamma_{\delta}$ (thus, $\nabla(\widetilde\w\wedge\n)=0$ on $\Gamma_{\delta}$ as well). We get, thanks to the usual decomposition $\mathbf a=(\mathbf a\cdot\mathbf n)\,\mathbf n+\mathbf n\wedge(\mathbf a\wedge\mathbf n)$, thanks to Lemma \ref{eps4} and to \eqref{KAPPA}, and by recalling that $\nabla\eta_{\eps,\lambda}=0$ outside $\tilde S_{\eps,\lambda}$ we get, \begin{equation*}\lab{convl2}\begin{aligned}
\displaystyle\int_{\om}|\widetilde\w\wedge\nabla \eta_{\eps,\lambda}|^2\,dx&\le\int_{\tilde S_{\eps,\lambda}}|\widetilde\w\wedge\nabla \eta_{\eps,\lambda}|^2\,dx
\\
&\le \displaystyle 4\int_{\tilde S_{\eps,\lambda}}|\nabla\eta_{\eps,\lambda}\cdot\n|^2|\widetilde\w\wedge\n|^2\,dx+4\int_{S^{**}_{\eps,\lambda}}  |\widetilde\w|^2\,|\n\wedge(\nabla\eta_{\eps,\lambda}\wedge\n)|^2\\
&\le\displaystyle \frac{C^2}{\lambda^2}\int_{S^{**}_{\eps,\lambda}}|\widetilde\w|^2\,dx+\frac{C^2}{\eps^2}\int_{S_{2\eps,2\lambda}}|\widetilde\w|^2\,dx\\
&\le\displaystyle \frac{C^2}{\lambda^2}\int_{S^{**}_{\eps,\lambda}}|\widetilde\w|^2\,dx+ 2C^2\int_{{S}_{2\eps,2\lambda}}|\nabla\widetilde\w|^2\,dx\\
\end{aligned}
\end{equation*}
so that $\widetilde\w\wedge\nabla \eta_{\eps,\lambda}\to 0$ in $L^2(\om,\mathbb R^3)$ as $\eps\to0^+$. 
On the other hand, still making use of Lemma \ref{eps4},
\begin{equation*}\lab{convl2der1}\begin{aligned}
&\displaystyle\int_{\om}|\partial_i\widetilde\w\wedge\nabla \eta_{\eps}|^2\,dx\le\int_{\tilde S_{\eps,\lambda}}|\partial_i\widetilde\w\wedge\nabla \eta_{\eps,\lambda}|^2\,dx\\
&\qquad\le \displaystyle 4\int_{\tilde S_{\eps,\lambda}}\left ( |\partial_i\widetilde\w|^2\,|\n\wedge\nabla\eta_{\eps,\lambda}\wedge\n|^2+|\nabla\eta_\eps\cdot\n|^2|\partial_i\widetilde\w\wedge\n|^2\right )\,dx\\
&\qquad\le \displaystyle \frac{C^2}{\lambda^{2}}\int_{\tilde S_{\eps,\lambda}}|\partial_i\widetilde\w|^2\,dx
+\frac{C^2}{\eps^2}\int_{S_{2\eps,2\lambda}}|\partial_i(\widetilde\w\wedge\n)|^2\,dx 
\displaystyle + \frac{C^2}{\eps^2}\int_{{S}_{2\eps,2\lambda}}|\widetilde\w\wedge\partial_i\n|^2\,dx
\\&\qquad
\le \displaystyle \frac{C^2}{\lambda^{2}}\int_{\tilde S_{\eps,\lambda}}|\partial_i\widetilde\w|^2\,dx
+\frac{C^2}{\eps^2}\int_{{S}_{2\eps,2\lambda}}|\nabla(\widetilde\w\wedge\n)|^2\,dx 
\displaystyle + \frac{K^2C^2}{\eps^2}\int_{{S}_{2\eps,2\lambda}}|\widetilde\w|^2\,dx
\\&\qquad
\le \displaystyle \frac{C^2}{\lambda^{2}}\int_{\tilde S_{\eps,\lambda}}|\partial_i\widetilde\w|^2\,dx
+{8C^2}\int_{{S}_{2\eps,2\lambda}}|\nabla^2(\widetilde\w\wedge\n)|^2\,dx 
\displaystyle + {8K^2C^2}\int_{{S}_{2\eps,2\lambda}}|\nabla\widetilde\w|^2\,dx
\\&\qquad\le \frac{C^2}{\lambda^{2}}\int_{\tilde S_{\eps,\lambda}}|\partial_i\widetilde\w|^2\,dx+ 24K^2\displaystyle C^2\int_{{S}_{2\eps,2\lambda}}(|\widetilde\w|^2+|\nabla \widetilde\w|^2+|\nabla^2 \widetilde\w|^2)\,dx
\end{aligned}
\end{equation*}
hence $\partial_i\widetilde\w\wedge\nabla \eta_{\eps,\lambda}\to 0$ in $L^2(\om,\mathbb R^3)$ as $\eps\to0^+$.
Similarly, taking advantage of the fact that $\curl\widetilde\w=0$ on $\Gamma_{\mu,\delta}$ and of Lemma \ref{eps4}, we get as $\eps\to 0^+$
\begin{equation*}\lab{convrot}
\displaystyle\int_{\om} |\partial_i\eta_{\eps,\lambda}\curl\widetilde\w|^2\,dx\le \frac{C^2}{\eps^2}\int_{\tilde{S}_{\eps,\lambda}}|\curl\widetilde\w|^2\,dx \le {8C^2} \int_{{S}_{2\eps,2\lambda}}|\nabla\mathrm{curl}\,\widetilde\w|^2\,dx\to 0,
\end{equation*}
thus $\partial_i\eta_{\eps,\lambda}\curl\widetilde\w\to 0$ in $L^2(\om,\mathbb R^3)$ as $\eps\to0^+$.
Moreover,  by applying the usual decomposition in normal and tangential part to $\nabla\eta_{\eps,\lambda}$ we get
\begin{equation*}\label{fourmore}\begin{aligned}
\widetilde\w\wedge\partial_i(\nabla\eta_{\eps,\lambda})&=
\partial_i(\nabla \eta_{\eps,\lambda}\cdot\mathbf n)\,\widetilde\w\wedge\n+(\nabla \eta_{\eps,\lambda}\cdot\mathbf n)\,\widetilde\w\wedge\partial_i\mathbf n\\&\qquad+\widetilde\w\wedge(\partial_i\mathbf n\wedge(\nabla \eta_{\eps,\lambda}\wedge\mathbf n))+\widetilde\w\wedge(\mathbf n\wedge\partial_i(\nabla \eta_{\eps,\lambda}\wedge\mathbf n))\\
&=\partial_i(\nabla \eta_{\eps,\lambda}\cdot\mathbf n)\,\widetilde\w\wedge\n+(\nabla \eta_{\eps,\lambda}\cdot\mathbf n)\,\widetilde\w\wedge\partial_i\mathbf n+\widetilde\w\wedge(\partial_i(\n\wedge(\nabla\eta_{\eps,\lambda}\wedge\n))),
\end{aligned}\end{equation*}
and by taking again advantage of \eqref{eta}  we find 
\[
\begin{aligned}
&\int_{\om}|\partial_i(\nabla\eta_{\eps,\lambda}\cdot\mathbf n)\,\widetilde\w\wedge\mathbf n|^2\,dx \le \int_{\tilde S_{\eps,\lambda}}|\nabla(\nabla\eta_{\eps,\lambda}\cdot\mathbf n)|^2\,|\widetilde\w\wedge\n|^2\,dx\le\frac{C^2}{4\eps^4}\int_{ S_{2\eps,2\lambda}}|\widetilde\w\wedge\mathbf n|^2\,dx,\\
&\int_\om|(\nabla\eta_{\eps,\lambda})\,\widetilde\w\wedge\partial_i\mathbf n|^2\,dx\le  K^2 \int_{\tilde S_{\eps,\lambda}} |\nabla\eta_{\eps,\lambda}\cdot\mathbf n|^2\,|\widetilde\w|^2\,dx\le\frac{K^2C^2}{4\eps^2}\int_{ S_{2\eps,2\lambda}}|\widetilde\w|^2\,dx,\\
&\int_\om\!|\widetilde\w\wedge(\partial_i(\n\wedge(\nabla\eta_{\eps,\lambda}\wedge\mathbf n)))|^2\,dx\le \!\int_{\tilde S_{\eps,\lambda}}\! |\nabla(\n\wedge(\nabla\eta_{\eps,\lambda}\wedge\mathbf n))|^2\,|\widetilde\w|^2\,dx\le\frac{C^2}{4\eps^2\lambda^2}\int_{ S_{2\eps,2\lambda}}\!|\widetilde\w|^2dx
\end{aligned}
\]
and we see that all these integrals are  reduced to the ones of the previous estimates, so that indeed by using  Lemma \ref{eps4} they all vanish as $\eps\to 0^+$, showing that  $\widetilde\w\wedge\partial_i(\nabla\eta_{\eps,\lambda})\to0$ in $L^2(\om,\mathbb R^3)$ as $\eps\to0^+$.
This proves the claim, so that \eqref{curlwtov} holds true.

 
 We stress that $\widetilde\w_{\eps,\lambda}$ and $\curl\widetilde\w_{\eps,\lambda}$ vanish in $S_{\eps,\lambda}$,
hence in an    open neighbor of $\Gamma$ in $\mathbb R^3$.
Let now $(\eps_j)_{j\in\mathbb N} \subset (0,+\infty)$ be  such that $\eps_j\to 0^+$ as $j\to +\infty$,  let $(\rho_j)_{j\in\mathbb N}$ be a sequence of smooth mollifiers such that the support of $\rho_j$ is so small that $\widetilde \w_{\eps_j,\lambda}\ast\rho_j$ still vanishes on a neighbor of $\Gamma$. Then, we define $\widetilde\w_{j,\lambda}:=\widetilde\w_{\eps_j,\lambda}\ast \rho_j$,  and $\v_j:=\curl\widetilde \w_{j,\lambda}$. It is readily seen that  $\v_{j}\in C^1(\overline{\om '},\R^3)$, that $\dv\, \v_{j}=0$ in $\om '$ and that  $\v_{j}=0$ on $\Gamma$. By recalling \eqref{curlwtov} we get also $\v_{j}\to \v$ in   $H^{1}(\om,\mathbb R^{3})$ thus concluding the proof. \end{proof}

\begin{remark}\label{C21}\rm
If $\Gamma=\partial\om$ (i.e., if $\partial\Gamma=\emptyset$), then it is enough to assume $C^{2,1} $ regularity of $\partial\om$ in Lemma \ref{gammal}. Indeed, we still get boundedness of $\nabla^2\n$ in $\om''$, which allows the latter proof to carry over.
\end{remark}
\KKK

The final step of this section is another suitable approximation property of divergence-free $H^1$ vector fields, that is required to treat the case $\partial\Gamma\neq\emptyset$. It is stated in Lemma \ref{GlvsG}. It also requires some  preliminary lemmas and  some further notation.

 Suppose that $\partial\Gamma\neq\emptyset$. For every $0< \mu< \mu_0$, where $\mu_0$ is defined by \eqref{mu0}, let
 \begin{equation*}\begin{array}{ll}
 \partial_l S_{\mu}:=\{\sigma+t\n(\sigma):\ \sigma\in \partial\Gamma,\ |t|\le \mu\}.
 \end{array}\end{equation*}
 Since $\n\in C^2(\partial\om)$, we see that $\partial_lS_\mu$ is a compact $C^{2}$ manifold with boundary. It is the {\it lateral boundary}  of $S_\mu$, and we denote by  $\bb\nu_l$  the corresponding outward unit  normal vector to $\partial_l S_{\mu}$.
 Moreover,  for every $0< \delta< \mu_0$ and $0<\mu<\mu_0$, 
  let $\bb\phi_{\mu,\delta}: \partial_l S_{\mu}\times [-\delta,\delta]\to \mathbb R^3$ be defined by \begin{equation}\label{fidelta}\bb\phi_{\mu,\delta}(s,\tau):= s+\tau\bb\nu_l(s).\end{equation}    It is clear that each point of $\bb\phi_{\mu,\delta}(\partial_l S_\mu\times[-\delta,\delta])$ is within the reach of both $\partial\om$ and $\partial\Gamma$. $\bb\phi_{\mu,\delta}$ inherits the $C^1$ regularity of $\bb\nu_l$.
 We recall the following simple property.
\begin{lemma}\lab{phidelta} Assume 
\eqref{OMEGA} and
\eqref{Gamma}, with $\partial\Gamma\neq\emptyset$. Let $\mu<\mu_0$.
There exists  $\delta_0\in(0,\mu_0)$ such that the map $\bb\phi_{\mu,\delta}:\partial_l S_\mu\times[-\delta_0,\delta_0]\to\mathbb R^3$, defined by  \eqref{fidelta}, is one-to-one.
\end{lemma}
\begin{proof} 
Assume by contradiction that there is no $\delta_0>0$ with the required property. Then there exist sequences $(s_j)_{j\in\mathbb N}\subset\partial_l S_\mu$, $(s'_j)_{j\in\mathbb N}\subset\partial_l S_\mu$ and $(t_j)_{j\in\mathbb N}\subset(0,1/n)$, $(t'_j)_{j\in\mathbb N}\subset(0,1/n)$ such that $(s_j,t_j)\neq (s'_j,t'_j)$ and
$\bb\phi_{\mu,\delta}(s_j,t_j)=\bb\phi_{\mu,\delta}(s'_j,t'_j)$ for any $j\in\mathbb N$ (hence, $s_j\neq s'_j$ for any $j\in\mathbb N$). Since $\partial_l S_\mu$ is compact, up to subsequences we have $s_j\to  s\in\partial_lS_\mu$,  $s'_j\to s'\in\partial_lS_\mu$, $t_j\to 0$ and $t'_j\to0$. Therefore, the continuity of $\bb\phi_{\mu,\delta}$ implies $\bb\phi_{\mu,\delta}(s,0)=\bb\phi_{\mu,\delta}(s',0)$, i.e., $s'=s$. This means that  $(s,0)$ has no open neighbor in $\mathbb R^3$ where $\bb\phi_{\mu,\delta}$ is invertible. Let $B$ denote the unit open ball in $\mathbb R^2$:  for a given $C^2$ local chart $u\ni B\mapsto {\bb\psi}(u)\in\mathbb R^3$ describing a neighbor of $s$ on the surface $\partial_l S_\mu$,  such that $\bb\psi(0)=s$, we have $|\det D\bb\phi_{\mu,\delta}(\bb\psi(0),0)|=|\partial_1\bb\psi(0)\wedge\partial_2\bb\psi(0)|\neq 0$ (by the regularity of the surface), where $D$ denotes the gradient in the variable $(u,t)$. Therefore, the $C^1(B\times(-\delta,\delta))$ mapping $\bb\phi_{\mu,\delta}\circ(\bb\psi,\bb{i})$ has non vanishing Jacobian at the point $(s,0)=\bb\phi_{\mu,\delta}(\bb\psi(0),0)$, hence it is invertible in a neighbor of such point,
a contradiction.
\end{proof}

 Still for $\partial\Gamma\neq\emptyset$, for every
$0< \mu< \mu_0$
 and  for every $0< \delta <  \delta_0$, where $\delta_0$ is the threshold provided by Lemma \ref{phidelta}, we define
\begin{equation}\lab{TSigma}\begin{array}{ll}
& T^\pm_{\mu,\delta}:=\{ s\pm\tau\bb\nu_l(s): s\in \partial_l S_{\mu},\ 0\le\tau \le\delta\},\qquad
  T_{\mu,\delta}:=T^+_{\mu,\delta}\cup T^-_{\mu,\delta}.
 \end{array}
 \end{equation}

We have the following

\begin{lemma}\label{lastlemma} 
Assume 
\eqref{OMEGA} and
\eqref{Gamma}, with $\partial\Gamma\neq\emptyset$.
Let $\w\in H^2(\om,\mathbb R^3)$ be such that $\w=\frac{\partial}{\partial\n} (\w\wedge\n)=0$ on $\Gamma$. Then there exist a vanishing sequence $(\lambda_j)_{j\in\mathbb N}\subset(0,\mu_0)$ and a sequence $(\w_{j})_{j\in\mathbb N}\subset H^2(\om,\mathbb R^3)$ such that $\w_j=\frac{\partial}{\partial\n} (\w_j\wedge\n)=0$ on $\Gamma_{\lambda_j}$ and $\w_{j}\to \w$ in $H^2(\om,\mathbb R^3)$ as $j\to+\infty$.
\end{lemma} 

\begin{proof} 
Let   $\delta_0$ be the threshold from Lemma \ref{phidelta}. 
  Let $\mu$ and $\delta$ be such that $$0<3\delta<\mu<\frac12\,\delta_0\qquad\mbox{and}\qquad 4\delta<\frac{1}{\mathrm{Lip}(\bb\nu_l)},$$
where 
$ \mathrm{Lip}(\bb\nu_l):=\sup\left\{\frac{|\bb\nu_l(z)-\bb\nu_l(z')|}{|z-z'|}:\;z\in\partial_lS_\mu,\; z'\in\partial_lS_\mu,\, z\neq z'\right\}<\!+\infty,$ since $\bb\nu_l\in C^1(\partial_lS_\mu)$.

 We define for every $y\in \Gamma_{2\delta}$ and for every $0< \lambda<  (\delta/2)\wedge1$ 
\begin{equation*} \psi_\lambda(y):=\left\{\begin{array}{ll}
y-2\gamma_\lambda(d^2(y,\partial\Gamma)){(\bf t\wedge \bf n)}(s(y))&\quad \hbox{if}\ y\in
\{x\in\Gamma_{2\delta}: d(x,\partial\Gamma)\le 2\delta\}\\
&\\
 y&\quad \hbox{otherwise in}\ \Gamma_{2\delta}\\
\end{array}\right.
\end{equation*}
where $s(y)\in \partial\Gamma$ is the unique nearest point of $\partial\Gamma$ to $ y\in \{x\in\Gamma_{2\delta}:\ d(x,\partial\Gamma)\le 2\delta\}$ (recalling \eqref{mu0} so that $2\delta<\delta_0<\mu_0$ implies that $y$ is within the reach of $\partial\Gamma$), and  ${\bf t}$ is the unit tangent vector to $\partial\Gamma$ (positively orienting $\partial\Gamma$ with respect to $\mathbf n$, so that $(\mathbf t\wedge\n)(s(y))$ coincides with the outward unit vector $\bb\nu_l(s(y))$ to $\partial_lS_{\mu}$). Moreover, here $\gamma_\lambda\in C^2(\mathbb R)$ is a decreasing  cutoff function such that
$\gamma_\lambda(\xi)= 0$ if $  \xi\ge \delta$ and
$\gamma_\lambda(\xi)= \lambda$ if  $\xi\le \delta/2$. We stress that $d^2(\cdot,\partial\Gamma)$ is a $C^3$ function on the set $\{x\in\mathbb R^3: d(x,\partial\Gamma)<\mu_0\}$, so that since $s(y)=y-\tfrac12\nabla(d^2(y,\partial\Gamma))$, $s(\cdot)$ is $C^2$ on such set (see Remark \ref{distancefunction}).

The following property holds: there exists $\lambda_0\in(0,(\delta/2)\wedge1)$ such that, for any $\lambda<\lambda_0$,  \begin{equation}\label{crucial}\sigma(\psi_\lambda(y))\in\Gamma\quad\mbox {for every $y\in \Gamma_\lambda$},\end{equation}
where $\sigma(\cdot)$ denotes as usual the unique projection  on $\partial\om$ (since $2\gamma_\lambda\le 2\lambda<\mu_0$, then $\psi_\lambda(y)$ is within the reach of $\partial\om$). This crucial property is proved in a separate statement, i.e., in Lemma \ref{auxiliar} below.

{

Let us consider an $H^2(\mathbb R^3, \mathbb R^3)$ extension of $\w$, still denoted by $\w$.
For any $x\in S_{2\mu,2\delta}$ (thus within the reach of $\partial\om$, since $2\mu<\delta_0<\mu_0$), 
we  introduce the signed distance function
\[
b(x):=\left\{\begin{array}{ll} d(x,\partial\om)\quad&\mbox{ if $x\notin\om$}\\
-d(x,\partial\om)\quad&\mbox{ if $x\in\om$,}\end{array}\right.
\]
and we notice that $b(x)$ is the unique real number such that $x=\sigma(x)+t(x)\n(\sigma(x))$. We notice that   $b\in C^3(\om'')$: indeed, we have $\n=\nabla b$.
See also \cite[Theorem 3.1]{DZ} for regularity results about the signed distance function. 
We let $g_\lambda:S_{2\mu,2\delta}\to\partial\om$ be defined by $g_\lambda(x):=\sigma(\phi_\lambda(\sigma(x)))$ and we further define
\begin{equation*}\lab{hl}\begin{aligned}
\mathbf h_\lambda(x):&=\left[\w(g_\lambda(x)+b(x)\n(g_\lambda(x)))\cdot\n(g_\lambda(x))\right]\;\n(\sigma(x))\\
&\qquad+\n(\sigma(x))\,\wedge\, \left[ \w(g_\lambda(x)+b(x)\n(g_\lambda(x)))\wedge\n(g_\lambda(x))\right].
\end{aligned}
\end{equation*}
We note that $\n$ is extended to a $C^2(\om'',\mathbb R^3)$ vector field  in the usual way, see \eqref{om''}, so that $\n(\sigma(x))=\n(x)$.
We let $ \eta\in C^{2}(\mathbb R^3)$ be a cutoff function, such that  $0\le \eta\le 1$,  $\eta\equiv 0\ \hbox{in}\  S_{\mu,\delta},$ and the support of $ 1-\eta$ is contained in $\mathring S_{2\mu,2\delta}$. Then we define, for every $x\in \mathbb R^3$,
\begin{equation*}\lab{wstarlambda}
\w^*_\lambda(x)=(1-\eta(x)) \mathbf h_\lambda(x),
\end{equation*}
so that indeed $\w_\lambda^*$ is supported in $\mathring S_{2\mu,2\delta}$, and
\begin{equation*}\lab{wlambda}
\w_\lambda(x):= \w^*_\lambda(x)+\eta(x) \w(x).
\end{equation*}

We claim that $\w_\lambda\in H^2(\om,\mathbb R^3)$ for any small enough $\lambda$ and that  $\w_\lambda\to \w$ in $H^2(\om,\mathbb R^3)$ as $\lambda\to 0$. Indeed, recalling the definition of $\gamma_\lambda$ and the fact that $d^2(\cdot,\partial\Gamma)$ and $s(\cdot)$ are $C^2$ functions on $\{x\in\mathbb R^3: d(x,\partial\Gamma)<\mu_0\}$, it is readily seen that $\psi_\lambda\in C^{2}(\Gamma_{2\delta})$ and that $\psi_\lambda(y)\to y$ in $C^{2}(\Gamma_{2\delta})$ as $\lambda\to0$. 
We also observe that $\sigma:\om''\to\partial\om$ is a $C^2$ function by assumption \eqref{OMEGA}. Therefore, we obtain $ \psi_\lambda\circ\sigma\to \sigma $ in $C^2(S_{2\mu,2\delta})$ as $\lambda\to 0$, and similarly $g_\lambda=\sigma\circ \psi_\lambda\circ\sigma\to\sigma$ in $C^2(S_{2\mu,2\delta})$ and $\n\circ g_\lambda\to \n\circ\sigma=\n$ in $C^2(S_{2\mu,2\delta})$. Since $x= \sigma(x)+b(x)\n(\sigma(x))$, for small enough $\lambda$ we see that the mapping $$S_{2\mu,2\delta}\ni x\mapsto q_\lambda(x):= g_\lambda(x)+b(x)\n(g_\lambda(x))\in \om''$$ is a $C^2$ homeomorphism whose Jacobian is bounded away from $0$, and moreover by the previous remarks it converges to the identity as $\lambda\to 0$ in $C^2(S_{2\mu,2\delta})$. Since $\w\in H^2(\mathbb R^3,\mathbb R^3)$, we obtain by the properties of Sobolev functions (as in Lemma \ref{eps4}), that $\w\circ q_\lambda\in H^2(\mathring S_{2\mu,2\delta},\mathbb R^3)$, and  moreover it is easy to check that
 $\w\circ q_\lambda\to\w$ in $H^2(\mathring S_{2\mu,2\delta},\mathbb R^3)$ as $\lambda\to 0 $. Taking the product with the smooth cutoff function $1-\eta$ (supported on $S_{2\mu,2\delta}$),  we deduce that $(1-\eta)(\w\circ q_\lambda)\to (1-\eta)\w$ in $H^2(\mathbb R^3,\mathbb R^3)$. Moreover, since we also have $\mathbf n\circ g_\lambda\to\n\circ\sigma=\n$ in $C^2(S_{2\mu,2\delta})$, we obtain
\[
\w^*_\lambda=(1-\eta)\mathbf h_\lambda\to(1-\eta)\left\{[\w\cdot\n]\,\n+\n\wedge\,[\w\wedge\n]\right\}=(1-\eta)\w
\]
in $H^2(\mathbb R^3,\mathbb R^3)$ as $\lambda\to 0$. 
Thus  $\w_\lambda\to \w$ in $H^2(\om,\mathbb R^3)$ as $\lambda\to 0$. The claim is proved.

We shall now prove that  $\w_\lambda=\frac{\partial}{\partial\n} (\w_\lambda\wedge\n)=0$ on $\Gamma_\lambda$ for any small enough $\lambda$. Indeed, let $\lambda<\lambda_0$ be small enough, such that $\w_\lambda\in H^2(\om,\mathbb R^3)$. If $x\in \Gamma_\lambda$, then by the property \eqref{crucial} we get $g_\lambda(x)\in\Gamma$, and since $b(x)=0$ and $\w=0$ on $\Gamma$ we directly obtain 
$\mathbf h_\lambda(x)=\w^*_\lambda (x)=0$ on $\Gamma_\lambda$. But $\eta(x)=0$ as well (because $\eta=0$ on $S_{\mu,\delta}$ and $\lambda<\delta/2$ so that $\Gamma_\lambda\subset\Gamma_{\delta}\subset S_{\mu,\delta}$), hence $\w_\lambda(x)$=0. 
On the other hand for every $x\in \Gamma_\lambda$, we have $g_\lambda(x)=\sigma(\psi_\lambda(x))$, and we have $\sigma(x+r\n(x))=x$, $b(x+r\n (x))=r$ when $|r|$ is small enough, therefore
\begin{equation*}
\begin{aligned}
\mathbf h_\lambda(x+r\n(x)):&=\left[\w(\sigma(\psi_\lambda(x))+r\n(\sigma(\psi_\lambda(x)))\cdot\n(\sigma(\psi_\lambda(x)))\right]\;\n(x)\\ 
&\qquad+\n(x)\wedge\left[ \w(\sigma(\psi_\lambda(x))+r\n(\sigma(\psi_\lambda(x))))\wedge\n(\sigma(\psi_\lambda(x)))\right]
\end{aligned}
\end{equation*}
and 
\begin{equation*}\begin{aligned}
&\mathbf h_\lambda(x+r\n(x))\wedge\n(x+r\n(x))=\mathbf h_\lambda(x+r\n(x))\wedge\n(x)\\
&\qquad\qquad=\left\{\n(x)\wedge\left[(\w(\sigma(\phi_\lambda(x))+r\n(\sigma(\phi_\lambda(x))))\wedge\n(\sigma(\phi_\lambda(x)))\right]\right\}\wedge\n(x).
\end{aligned}
\end{equation*}
As a consequence, by taking into account that  $\sigma(\psi_\lambda(x))\in \Gamma$ and that $\frac{\partial}{\partial\n}(\w\wedge\n)=0$ on $\Gamma$ we get
\begin{equation*}\begin{aligned}
\displaystyle\frac{\partial}{\partial\n}(\mathbf h_\lambda\wedge\n)(x)&=\lim_{r\to 0^+}\frac{\mathbf h_\lambda(x+r\n(x))\wedge\n(x+r\n(x))-\mathbf h_\lambda(x)\wedge\n(x)}{r}\\
&\displaystyle=\n(x)\wedge\left(\frac{\partial (\w\wedge\n)}{\partial\n}(\sigma(\psi_\lambda(x)))
\wedge\n(x)\right)=0,
\end{aligned}
\end{equation*}
and since we have already shown that $\mathbf h_\lambda$ vanishes on $\Gamma_\lambda$, we conclude that
 $\frac{\partial}{\partial\n} (\w^*_\lambda\wedge\n)=0$ on $\Gamma_\lambda$. Again, $\eta$ is vanishing on $S_{\mu,\delta}$, hence in an open neighbor of $\Gamma_\lambda$. We deduce that $\frac{\partial}{\partial\n}(\w_\lambda\wedge\n)=0$ on $\Gamma_\lambda$.

Eventually, by taking a vanishing sequence of small enough positive numbers $(\lambda_j)_{j\in\mathbb N}$, we conclude that $\w_j:=\w_{\lambda_j}$ satisfies all the desired properties.
}
\end{proof}

\begin{lemma}\label{auxiliar}
Assume 
\eqref{OMEGA} and
\eqref{Gamma}, with $\partial\Gamma\neq\emptyset$.
Let $\mu,\delta,\lambda, \phi_\lambda$  as in the proof of {\rm Lemma \ref{lastlemma}}. There exists $\lambda_0\in(0,(\delta/2)\wedge1)$ such that, for any $\lambda<\lambda_0$, \eqref{crucial} holds.
\end{lemma}
\begin{proof}
If $d(y,\partial\Gamma)\ge \delta$ there is nothing to prove, since in this case either $y=\phi_\lambda(y)\in\Gamma$ or $y\notin \Gamma_\lambda$ (because $\lambda<\delta/2$).
 Therefore, we assume from now on that $d(y,\partial\Gamma)<\delta$ and we prove the result in four steps.\\
 
\noindent\textbf{Step 0}. We start by showing the elementary properties  $T^{-}_{\mu,3\delta}\subset S_{2\mu}$ and $T^+_{\mu,3\delta}\cap \mathring S_{2\mu}=\emptyset$.  Indeed, concerning the first property, we may prove that for any $s\in\partial_lS_\mu$, a point of the form $s-\alpha\bb\nu_l(s)$, with $0\le\alpha\le 3\delta$, belongs to $S_{2\mu}$. This is obvious if   $\alpha$ is small, since $\bb\nu_l$ is normal to $\partial S_{2\mu}$. Moreover, as $\alpha$ increases without reaching the threshold $\delta_0$, the closest point of $\partial_lS_{2\mu}$ is always $s$, by Lemma \ref{phidelta} which gives the unique projection property on $\partial_lS_{2\mu}$. This shows that no other point of $\partial_lS_{2\mu}$ can be a reached. And since $\alpha\le3\delta$, we are also far from $\partial S_{2\mu}\setminus\partial_lS_{2\mu}\subset\{x\in\mathbb R^3: d(x,\partial\om)=2\mu\}$, because $d(s-\alpha\bb\nu_l(s),\partial\om)\le d(s,\partial\om)+3\delta=\mu+3\delta<2\mu$. The second property is proved in the same way. \\

\noindent \textbf{Step 1}. We  check that 
  $y\in T_{\mu,\delta}$. 
  Indeed, we have $d(y,\partial_lS_\mu)\le d(y,\partial\Gamma)=d(y,s(y))<\delta$. We take a point $s_*\in\partial_lS_\mu$ such that $|y-s_*|=d(y, \partial_lS_\mu)$, therefore $d(s_*,s(y))\le d(y,s(y))+d(y,s_*(y))<2\delta<\mu$ so that $s_*$ is not on the boundary of $\partial_lS_\mu$.  Since  $s_*$ is a minimizer of the distance function from the $C^2$ manifold $\partial_lS_\mu$, the corresponding first order minimality conditions immediately imply that $y-s_*$ is orthogonal to the tangent plane to $\partial_lS_\mu$ at $s_*$, so that  
  \begin{equation*}\label{eins}
  y=s_*+\tau_*(y)\bb\nu_l(s_*), 
  \end{equation*}
  where $|\tau_*(y)|=d(y,\partial_lS_\mu)<\delta$. Thus, $y\in T_{\mu,\delta}$.
  We notice that by Lemma \ref{phidelta}, $s_*$  coincides in fact with  the unique  projection  $s_*(y)$ of $y$ on $\partial_lS_\mu$.
  \\

   \noindent \textbf{Step 2}.
We prove the result in the case $y\in \Gamma_\lambda\setminus\Gamma$. 
In this case we have $d(y,\partial\Gamma)=d(y,s(y))\le\lambda$ and as a consequence $d(s(y),s_*(y))\le 2\lambda$. We also have $\gamma_\lambda(d^2(y,\partial\Gamma))=\lambda$ because $\lambda<(\delta/2)\wedge 1$,  and moreover using Step 1 and $y\notin\Gamma$ we have $y\in T^+_{\mu,\delta}$, so that $\tau_*(y)>0$. Actually, $y\in T^+_{2\lambda,\delta}$ as well, since $d(s(y),s_*(y))\le 2\lambda$.

By taking into account that $(\mathbf t\wedge \mathbf n)(s(y))=\bb\nu_l(s(y))$ we get
\begin{equation}\label{zwei}\begin{aligned}
\psi_\lambda(y)&=y-2\gamma_\lambda(d^2(y,\partial\Gamma))(\mathbf t\wedge \mathbf n)(s(y))=y-2\lambda\bb\nu_l (s(y))\\&=s_*(y)+\tau_*(y)\bb\nu_l(s_*(y))-2\lambda\bb\nu_l (s(y))\\
&=s_*(y)+(\tau_*(y)-2\lambda)\bb\nu_l (s_*(y))+2\lambda (\bb\nu_l (s_*(y))-\bb\nu_l (s(y))).
\end{aligned}
\end{equation}
We notice that the point $s_*(y)+(\tau_*(y)-2\lambda)\bb\nu_l (s_*(y))$ belongs to $T^{-}_{2\lambda,\delta}$, because $\tau_*(y)=d(y,\partial_lS_\mu)\le d(y,s(y))\le \lambda$ and therefore $-\delta<-2\lambda\le \tau_*(y)-2\lambda\le-\lambda$, and we have in particular
\begin{equation}\label{drei}
d(s_*(y)+(\tau_*(y)-2\lambda)\bb\nu_l (s_*(y)), \partial_lS_\mu)\ge \lambda.\end{equation}
But
\begin{equation}\label{vier}
|2\lambda (\bb\nu_l (s_*(y))-\bb\nu_l (s(y)))|\le 2\lambda \,\mathrm{Lip}(\bb\nu_l)\,|s_*(y)-s(y)|\le 4\lambda^2 \,\mathrm{Lip}(\bb\nu_l).
\end{equation}
By \eqref{zwei}, \eqref{drei} and \eqref{vier}, there exists a small enough $\lambda_0$  depending on  $ \mathrm{Lip}(\bb\nu_l)$ such that for any $\lambda<\lambda_0$ we have $\psi_\lambda(y)\in T^-_{\mu,\delta}$. Since $T^-_{\mu,\delta}\subset S_{2\mu}$ by Step 0, the result is proved.\\

%

 \noindent \textbf{Step 3}. 
We prove the result in case $y\in\Gamma$. Since $y\in T_{\mu,\delta}$ by Step 1, we have in this case
$y\in T^-_{\mu,\delta}$, therefore we have $\tau_*(y)=-d(y, \partial_lS_\mu)$. By  $d(s_*(y), s(y))<2\delta$, we have in particular, $y\in T^-_{2\delta,\delta}$.
With the same computation of Step 2, we obtain an expression which is analogous to \eqref{zwei}, that is,
\begin{equation}\label{funf}
\psi_\lambda(y)=s_*(y)-(d(y,\partial_lS_\mu)+2\gamma_\lambda(d^2(y,\partial\Gamma)))\bb\nu_l (s_*(y))+2\gamma_\lambda(d^2(y,\partial\Gamma)) (\bb\nu_l (s_*(y))-\bb\nu_l (s(y))).
\end{equation}
In particular, since $0\le d(y,\partial_lS_\mu)+2\gamma_\lambda(d^2(y,\partial\Gamma))\le \delta+2\lambda<2\delta$, we have \begin{equation}\label{sechs}s_*(y)-(d(y,\partial_lS_\mu)+2\gamma_\lambda(d^2(y,\partial\Gamma)))\bb\nu_l (s_*(y))\in T^{-}_{2\delta,\delta+2\lambda},\end{equation} with
\begin{equation}\label{sieben}
d(s_*(y)-(d(y,\partial_lS_\mu)+2\gamma_\lambda(d^2(y,\partial\Gamma)))\bb\nu_l (s_*(y)),\partial_lS_\mu)\ge 2\gamma_\lambda(d^2(y,\partial\Gamma)).
\end{equation}
 By the assumptions on $\delta$, we have
 \begin{equation}\label{acht}
 |\bb\nu_l (s_*(y))-\bb\nu_l (s(y))|\le \mathrm{Lip}(\bb\nu_l)|s_*(y)-s(y)|\le 2\delta\,\mathrm{Lip}(\bb\nu_l)\le \frac12.
 \end{equation}
By \eqref{funf}, \eqref{sechs}, \eqref{sieben} and \eqref{acht}, we conclude that $\psi_\lambda(y)\in T^-_{2\delta+\lambda,\delta+3\lambda}$. Therefore, $y\in T^{-}_{\mu,3\delta}$, since $2\delta+\lambda<\mu$ and $\delta+3\lambda<3\delta$.  By Step 0,   the result  is proved.
 \end{proof}

Thanks to the results of Section \ref{Sectionpotential} and to Lemma \ref{lastlemma}, we  deduce the final  approximation result for curl vector fields.

\begin{lemma} \lab{GlvsG}
 Assume 
\eqref{OMEGA} and
\eqref{Gamma}, with $\partial\Gamma\neq\emptyset$.
Let $ \v\in H^{1}(\om,\mathbb R^{3})$ such that $\mathrm{div}\, \v=0$ a.e. in $\om$ and $\v=0$ on $\Gamma$. Then there exist a vanishing sequence  $ (\lambda_j)_{j\in\mathbb N}\subset(0,\mu_0)$ and a sequence $(\v_j)_{j\in\mathbb N}\subset H^{1}(\om,\mathbb R^{3})$ such that  $\mathrm{div}\, \v_j=0$ a.e. in $\om$, $\v_j=0$ on $\Gamma_{\lambda_j}$ and such that $\v_j\to\v$ in $H^{1}(\om,\mathbb R^{3})$ as $j\to+\infty$.
\end{lemma}
\begin{proof} 
By Lemma \ref{w=0} there exists $\widetilde\w\in H^{2}(\Omega,\R^3)$ such that $\widetilde\w=0$ on $\Gamma$ and $\hbox{curl}\, \widetilde\w=\v$ a.e. in $\om$, so that Lemma \ref{dernormwtang} implies $\frac{\partial}{\partial \n}(\widetilde\w\wedge\n)=0$ on $\Gamma$.  Hence, by Lemma \ref{lastlemma} there exist a vanishing sequence $(\lambda_j)_{j\in\mathbb N}\subset(0,+\infty)$ and a sequence $(\w_j)_{j\in\mathbb N}\subset H^2(\om,\mathbb R^3)$ such that $\w_j=\frac{\partial}{\partial\n} (\w_j\wedge\n)=0$ on $\Gamma_{\lambda_j}$ and $\w_j\to \widetilde\w$ in $H^2(\om,\mathbb R^3)$. By Lemma \ref{dernormwtang} we get $\curl\w_j=0$ on $\Gamma_{\lambda_j}$, hence by setting $\v_j:=\curl\w_j$ the result follows.
\end{proof}

%
%

 \section{Proof of the main result}

Let us start by recalling the following version of the rigidity inequality by Friesecke, James and M\"uller \cite{FJM0}.
\begin{lemma} {\rm ({\bf Geometric Rigidity Inequality} \cite{FJM}, \cite{ADMDS}).} Let $g_{p}$ the function defined in \eqref{gp}. There exists a constant $C_p=C_{p}(\om) >0$ such that for every $\yy\in W^{1,p}(\om,\mathbb R^{3})$
there exists a constant $\mathbf R\in SO(3)$ such that we have
\begin{equation}\label{muller}
\int_{\om}g_{p}(|\nabla\yy-\mathbf R|)\,dx\le C_{p}\int_{\om}g_{p}(d(\nabla \yy, SO(3)))\,dx.
\end{equation}
\end{lemma}

Based on the above result, we deduce compactness of minimizing sequences, which follows in fact from the results in \cite{ADMDS}.

\begin{lemma}
\label{add} Assume \eqref{OMEGA}, \eqref{Gamma}, \eqref{framind}, \eqref{Z1}, \eqref{reg},    
\eqref{coerc}. 
Let $(h_j)_{j\in\mathbb N}$ be a sequence of positive real numbers  and let $(\v_j)\subset W^{1,p}(\om,\mathbb R^3)$ be a sequence such that  $\v_j=0$ on $\Gamma$ for any $j\in\mathbb N$. For every $j\in\mathbb N$,  let $\mathbf y_j=\bb{i}+h_j\v_j$  and let $\mathbf R_j\in SO(3)$ be a constant rotation satisfying  \eqref{muller}. Then there exists a constant $C>0$ (only depending on $p$, $\Omega$ and $\Gamma$) such that for any $j\in\mathbb N$ there hold
\begin{equation}\label{borrow}
|\mathbf I-\mathbf R_j|^2\le  C\,\int_\Omega \mathcal W^I(x, \mathbf I+h_j \v_j)\,dx
\end{equation}
and
\begin{equation}\label{bsec}
\int_\om |\nabla\v_j|^p\,dx\le  C\left(1+\int_\om\mathcal W^I(x, \mathbf I+h_j\nabla\v_j)\,dx\right).
\end{equation}
If we assume in addition that $h_j\to 0$ as $j\to+\infty$ and that
\begin{equation}\label{minus}
\lim_{j\to+\infty}\left(\mathcal G^I_{h_j}(\v_j)-\inf_{W^{1,p}(\om,\mathbb R^3)}\mathcal G^I_{h_j}\right)=0,
\end{equation}
then $\sup_{j\in\mathbb N}\|\v_j\|_{W^{1,p}(\om,\mathbb R^3)}<+\infty$.
\end{lemma}
\begin{proof}
We have $\mathcal W^I\ge \mathcal W$, and
\eqref{borrow} holds true with $\mathcal W$ in place of $\mathcal W^I$ as proven in \cite[Lemma 3.3]{ADMDS}  (by taking advantage of  assumption \eqref{coerc} on $\mathcal W$). Therefore, \eqref{borrow} holds.


 Using the form of $g_p$ 
it is clear that there exists a constant $c$ (only depending on $p$) such that
\[
\int_\om g_p(h_j|\nabla\v_j|)\,dx\le c\int_\om\left(g_p(|\mathbf I+h_j\nabla\v_j-\mathbf R_j|)+|\mathbf I-\mathbf R_j|^2\right)\,dx.
\]
Hence, by invoking the  rigidity estimate \eqref{muller}, there is another constant $K$ (only depending on $p$ and $\om$) such that
\[
\int_\om g_p(h_j|\nabla\v_j|)\,dx\le K\left(\int_\om g_p(d(\mathbf I+h_j\nabla \v_j, SO(3)))\,dx+ |\mathbf I-\mathbf R_j|^2 \right),
\]
and 
since  $x^p\le1+2g_p(x)$ holds for $x\ge 0$,  by making use of \eqref{coerc} and \eqref{borrow} it follows that there is a further constant $C$ (only depending on $\Omega$, $\Gamma$, p) such that 
\[
\int_\om |\nabla\v_j|^p\,dx\le \int_\om (1+2g_p(h_j|\nabla\v_j|))\,dx\le C\left(1+\int_\om\mathcal W(x, \mathbf I+h_j\nabla\v_j)\,dx\right).
\]
Since $\mathcal W\le \mathcal W^I$, 	\eqref{bsec} follows.

Let us prove the last statement.
Assuming \eqref{minus} and assuming wlog that $\|\v_j\|_{W^{1,p}(\om,\mathbb R^3)}\ge 1$ for any $j\in\mathbb N$, we get for any large enough $j$
\[
\frac1{h_j^2}\int_\om \mathcal W^I(x,\mathbf I+h_j\nabla\v_j)\,dx-\mathcal L(\v_j)=\mathcal G^I_{h_j}(\v_j)\le \mathcal G^I_{h_j}(\mathbf 0)+1=1,
\]
thus \eqref{bsec} implies
\begin{equation*}\
\int_\om |\nabla\v_j|^p\,dx\le  C\left(1+\int_\om\mathcal W^I(x, \mathbf I+h_j\nabla\v_j)\,dx\right)
\le C+C(h_j^2+C_{\mathcal L} 	\,h_j^2 \,\|\v_j\|^p_{W^{1,p}(\om,\mathbb R^3)}).
\end{equation*}
Since $h_j$ goes to zero, the result follows by Friedrichs inequality.
\end{proof}

We next prove $\Gamma$-convergence. The limsup inequality is based on the approximation results from Section \ref{approximationsection}. The liminf inequality builds on previous arguments from \cite{ADMDS, DMPN,  MPTARMA}. 

%

\begin{lemma} {\rm ({\bf Energy convergence}).}
\label{convfunc}
Assume \eqref{OMEGA}, \eqref{Gamma}, \eqref{framind}, \eqref{Z1}, \eqref{reg},    
\eqref{coerc}.  Let $(h_j)_{j\in\mathbb N}$ be a vanishing sequence of positive numbers.
Then the sequence of functionals $(\mathcal G_{h_j}^I)_{j\in\mathbb N}$ is $\Gamma$-converging to functional $\mathcal G^I$ with respect to the weak topology of $W^{1,p}(\om,\mathbb R^3)$.
\end{lemma}
\begin{proof} Since the weak topology of $W^{1,p}$ is metrizable then we can characterize the $\Gamma-$limit in terms of weakly converging sequences. In particular, by setting (see \cite{DM}, \cite{DGFra})
\begin{equation*}\lab{Gammalim}\begin{array}{ll}
&\displaystyle\mathcal G^I_- (\v):=\inf \{\liminf_{j\to \infty} \mathcal G_{h_j}^I(\v_j): \v_j\wconv \v \ \hbox{ weakly in} \ W^{1,p}(\om,\mathbb R^3)\},\\
&\\
&\displaystyle\mathcal G^I_+ (\v):=\inf \{\limsup_{j\to \infty} \mathcal G_{h_j}^I(\v_j): \v_j\wconv \v \ \hbox{ weakly in} \ W^{1,p}(\om,\mathbb R^3)\},\\
\end{array}
\end{equation*}
since  $\mathcal G^I_+ (\v)\ge \mathcal G^I_- (\v)$,  it is enough  to prove that $\mathcal G^I_+ (\v)\le \mathcal G^I(\v) \le \mathcal G^I_- (\v)$ for every $\v\in W^{1,p}(\om, \mathbb R^3)$. We split the proof in two steps.\\

{\bf Step 1 (liminf)}.  We show  that $\mathcal G^I(\v) \le \mathcal G^I_- (\v)$ for every $\v\in W^{1,p}(\om, \mathbb R^3)$. 

Let $\v\in W^{1,p}(\om, \mathbb R^3)$, assume without restriction that $\mathcal G^I_- (\v)< +\infty$, and let $(\v_j)_{j\in\mathbb N}\subset W^{1,p}(\om,\mathbb R^3)$ be a sequence such that $\v_{j}\wconv \v$ weakly in $W^{1,p}(\Omega,\R^3)$ as $j\to+\infty$ and such that $\sup_{j\in\mathbb N}\mathcal G_{h_{j}}^{I}(\v_j)<+\infty$. Then  $\v_{j}=0$ on $\Gamma$ for any $j\in\mathbb N$, hence $\v=0$ on $\Gamma$ as well, and by setting $\mathbf B_{j}:=2\mathbb E(\v_{j})+h_{j}\nabla\v_{j}^{T}\nabla\v_{j}$ we get 
\begin{equation*}\begin{aligned}
1&=\det(\mathbf I+h_{j}\nabla\v_{j})=\det(\mathbf I+h_{j}\nabla\v_{j}^T)(\mathbf I+h_{j}\nabla\v_{j})=\det(\mathbf I+2h_{j}\mathbb E(\v_{j})+h_{j}^{2}\nabla\v_{j}^{T}\nabla\v_{j})\\
&= 1+h_{j}\mathrm{Tr} \mathbf B_{j}-\frac{1}{2}h_{j}^{2}(\mathrm{Tr}(\mathbf B_{j}^{2})-(Tr \mathbf B_{j})^{2})+h_{j}^{3}\det \mathbf B_{j}
\end{aligned}
\end{equation*}
a.e. in $\om$, that is, 
\begin{equation}\lab{TrB}
Tr \mathbf B_{j}=2\dv \v_{j}+h_{j}|\nabla\v_{j}|^{2}=\frac{1}{2}h_{j}(Tr(\mathbf B_{j}^{2})-(Tr \mathbf B_{j})^{2})-h_{j}^{2}\det \mathbf B_{j}.
\end{equation}

We next prove, with an argument from \cite{ADMDS}, that $\v\in H^1(\om,\mathbb R^3)$ and that 
\begin{eqnarray}
&1_{D_j}\nabla\v_j\wconv \nabla\v\; \hbox{ weakly in}\ L^2(\om,\mathbb R^{3\times3}),\lab{dj1}\\
&1_{\om\setminus D_j}\nabla\v_j\to 0\ \hbox{ in}\ L^\alpha(\om,\mathbb R^{3\times3}),\;\; \forall \alpha\in [1,p),\lab{dj2}
\end{eqnarray}
where we have set
$
\displaystyle D_j:=\left\{x\in\om: \sqrt{h_j}\,|\nabla\v_j(x)|\le1\right\}.
$
Indeed, 
since we are assuming that $\sup_{j\in\mathbb N}\mathcal G^I_{h_j}(\v_j)<+\infty $, we have 
\begin{equation}\label{boundedwi}
\sup_{j\in\mathbb N}\frac1{h_j^2}\int_\om \mathcal W^I(x,\mathbf I+h_j\nabla\v_j)\,dx<+\infty,
\end{equation}
 thanks to the definition of $\mathcal G^I_{h_j}$ and to the boundedness of the sequence $(\v_j)_{j\in\mathbb N}$ in $W^{1,p}(\om,\mathbb R^3)$. Let $\mathbf R_j\in SO(3)$ be a constant matrix satisfying \eqref{muller} with respect to $\mathbf y_j=\bb i+h_j\v_j$.
If $Q_j:=\{x\in\om: |\mathbf I+h_j\nabla\v_j(x)-\mathbf R_j|\le 3\sqrt{3}\}$, we have  $D_j\subset Q_j$ for any $j$ large enough, and by definition of $g_p$ it is clear that there exists a constant $K$ only depending on $p$ such that $g_p(x)\ge Kx^2$ for any $x\in[0,3\sqrt{3}]$, so that
\[\begin{aligned}
\int_{D_j}|\nabla \v_j|^2\,dx&\le \frac K{h_j^2}\int_{Q_j} \left(g_p(|\mathbf I+h_j\nabla\v_j-\mathbf R_j|)+|\mathbf I-\mathbf R_j|^2\right)\,dx\\&\le \frac{KC}{h_j^2}\int_\om\mathcal W^I(x,\mathbf I+h_j\nabla\v_j)\,dx+K|\om|\,\frac{|\mathbf I-\mathbf R_j|^2}{h_j^2},
\end{aligned}\]
where we have used \eqref{coerc} and \eqref{muller}. By taking advantage of \eqref{borrow} and of \eqref{boundedwi}, we conclude that the sequence $(1_{D_j}\nabla \v_j)_{j\in\mathbb N}$ is bounded in $L^2(\om, \mathbb R^{3\times3})$, so that up to (not relabeled) subsequences, $ 1_{D_j}\nabla \v_j\rightharpoonup\mathbf H$ weakly in  $L^2(\om,\mathbb R^{3\times3})$.
On the other hand, if $\alpha\in[1,p)$,  by H\"older inequality and the definition of $D_j$ we have \[\begin{aligned}\|1_{\om\setminus D_j}\nabla\v_j\|_{L^\alpha(\om,\mathbb R^{3\times3})}\le\|\nabla\v_j\|_{L^p(\om,\mathbb R^{3\times 3})}\,|\om\!\setminus\! D_j|^{\frac{p-\alpha}{p\alpha}}
\!\le \|\nabla\v_j\|_{L^p(\om,\mathbb R^{3\times 3})}\!\left(\!\sqrt{h_j}\int_\om|\nabla\v_j|\,dx\!\right)^{\!\!\frac{p-\alpha}{p\alpha}}
\end{aligned}\]
and \eqref{dj2} follows from the fact that the above  right hand side is vanishing as $j\to+\infty$, since the sequence $(\v_j)_{j\in\mathbb N}$ is bounded in $W^{1,p}(\om,\mathbb R^3)$. The latter property also implies the weak convergence (up to not relabeled subsequences) of $\nabla\v_j$ to $\nabla \v$ in $L^{\alpha}(\om,\mathbb R^{3\times3})$:
 since  \eqref{dj2} holds and since $1_{D_j}\nabla \v_j=(\nabla\v_j-1_{\om\setminus D_j}\nabla\v_j)$, we obtain both $\nabla\v=\mathbf H\in L^2(\om,\mathbb R^{3\times3})$ and \eqref{dj1}, and Friedrichs inequality yields $\mathbf v\in H^1(\om,\mathbb R^3)$.

Thanks to the properties \eqref{dj1} and \eqref{dj2} we get 
\begin{equation*}
\sqrt{h_j}\nabla\v_j=\sqrt{h_j}(1_{D_j} \nabla\v_j+1_{\om\setminus D_j}\nabla\v_j)\to 0 \quad\mbox{ in $L^\alpha(\om,\mathbb R^{3\times3})$}
\end{equation*}
as $j\to+\infty$ for any $\alpha\in[1,p)$, hence (up to not relabeled subsequences) $\sqrt{h_j}\nabla\v_j\to 0$ a.e. in $\om$.
By taking into account that for some constant $c> 0$ there hold
\begin{equation*}\lab{TrB2}\begin{array}{ll}
&|Tr \mathbf B_j^2|\le c(|\nabla\v_j|^2+h_j^2|\nabla\v_j|^4+ h_j|\nabla\v_j|^3),\\
&\\
&|Tr \mathbf B_j|^2\le c(|\nabla\v_j|^2+h_j^2|\nabla\v_j|^4),\\
&\\
&|\det\mathbf B_j|\le c|\mathbf B_j|^3\le C(|\nabla\v_j|^3+h_j^3|\nabla\v_j|^6),\\
\end{array}
\end{equation*}
we get 
\begin{equation*}
h_{j}|\nabla\v_{j}|^{2}+\frac{1}{2}h_{j}(Tr(\mathbf B_{j}^{2})+(Tr \mathbf B_{j})^{2})+h_{j}^{2}\det \mathbf B_{j}\to 0
\end{equation*}
a.e. in $\om$ as $j\to+\infty$. Hence, by \eqref{TrB},  $\dv\v_j\to 0$ a.e. in $\om$ and by recalling that $\dv\v_j\wconv\dv\v$ weakly in $L^p(\om)$
we have $\dv\v=0$ a.e. in $\om$.  Since we have previously shown that $\v\in H^1(\om,\mathbb R^3)$ and that $\v=0$ on $\Gamma$, we deduce that $\mathcal G^I(\v)$ is finite.\KKK 

By  assumption \eqref{reg}, $D^2\mathcal W(x,\cdot)\in C^2(\mathcal U)$ for a.e. $x\in\om$ and    there is an increasing function $\omega:[0,+\infty)\to\mathbb R$ such that $\lim_{y\to0}\omega(y)=0$ and $|D^2\mathcal W(x,\mathbf I+\mathbf F)-D^2\mathcal W(x,\mathbf I)|\le \omega(|\mathbf F|)$ for any $\mathbf F\in \mathcal U$ and for a.e. $x\in\om$. We notice that for any large enough $j$, we have $\mathbf I+h_j\nabla\v_j\in\mathcal U$ for any $x\in D_j$.     Therefore, 
\begin{equation}\label{nee}\begin{aligned}
&\limsup_{j\to+\infty}\int_{D_j} \left|\frac{1}{h_j^2}\mathcal W(x,\mathbf I+h_j\nabla\v_j)-\frac12\,\nabla\v_j^T D^2\mathcal W(x,\mathbf I) \nabla\v_j\right|\,dx\\
&\qquad\le\limsup_{j\to+\infty}\int_{D_j}\omega(h_j|\nabla\v_j|)\,|\nabla\v_j|^2\,dx\le\limsup_{j\to+\infty}\,\omega(\sqrt{h_j})\int_{\om} 1_{D_j}|\nabla\v_j|^2\,dx =0,
\end{aligned}\end{equation}
where we have also used \eqref{smooth0} and \eqref{dj1}.

Finally, by taking advantage of \eqref{nee} and \eqref{dj1},  since $\mathcal W^I\ge\mathcal W$
and since the map $\mathbf F\mapsto\int_\om  \mathbf F^TD^2\mathcal W(x,\mathbf I)\,\mathbf F\,dx$ is lower semicontinuous with respect to the weak $L^2(\om,\mathbb R^{3\times3})$ convergence,
 we conclude that
\[\begin{aligned}
&\liminf_{j\to+\infty} \int_\om \frac1{h_j^2}\mathcal W^I(x,\mathbf I+h_j\nabla\v_j)\,dx\ge \liminf_{j\to+\infty}\int_{D_j}\frac{1}{h_j^2}\mathcal W(x,\mathbf I+h_j\nabla\v_j)\,dx
\\&\quad
\ge\liminf_{j\to+\infty}\int_{D_j}\frac12 \nabla\v_j^T D^2\mathcal W(x,\mathbf I)\nabla\v_j\,dx
=\liminf_{j\to+\infty}\int_{\om}\frac12(1_{D_j}\nabla\v_j)^T D^2\mathcal W(x,\mathbf I)(1_{D_j}\nabla\v_j)
\\&\quad\ge \int_\om\frac12\nabla\v^T D^2\mathcal W(x,\mathbf I)\nabla\v\,dx=\frac12\int_\om\mathbb E(\v)D^2\mathcal W(x,\mathbf I )\,\mathbb E(\v)\,dx.
\end{aligned}
\]
Since functional $\mathcal L$ from \eqref{external} is continuous with respect to the weak convergence in $W^{1,p}(\om,\mathbb R^3)$, we  get 
\[\begin{aligned}\liminf_{j\to+\infty}\mathcal G^I_{h_j}(\v_j)&=
\liminf_{j\to+\infty} \int_\om  \frac{1}{h_j^2}\mathcal W^I(x,\mathbf I+h_j\nabla\v_j)\,dx-\mathcal L(\v_j)\\&\ge\frac12\int_\om\mathbb E(\v)D^2\mathcal  W(x,\mathbf I)\,\mathbb E(\v))\,dx
-\mathcal L(\v)=\mathcal G^I(\v).
\end{aligned}\]
Therefore,   $\mathcal G^I_- (\v)< +\infty$  only if $\v\in H^1_{\mathrm{div}}(\om,\mathbb R^3)$ with $\v=0$ on $\Gamma$,  and  $\mathcal G^I(\v) \le \mathcal G^I_- (\v)$ for every $\v\in W^{1,p}(\om,\mathbb R^3)$.\\

{\bf Step 2 (limsup)}. We show now that $\mathcal G^I_+ (\v)\le \mathcal G^I(\v) $ for every $\v\in W^{1,p}(\om,\mathbb R^3)$.

 It will be enough to prove the inequality for every $\v\in H^1_{\mathrm{div}}(\om,\mathbb R^3)$ such that $\v=0$ on $\Gamma$ (otherwise $\mathcal G^I(\v)=+\infty$).  This will be done in three subsequent steps, that make use of Lemma \ref{reynolds}, Lemma \ref{gammal} and Lemma \ref{GlvsG}, respectively.  
 \KKK

Assume first that $\v$ is the restriction to $\om$ of a function $\v\in C^1(\om',\mathbb R^3)$ such that $\v=0$ on $\Gamma$ and $\dv\v=0$ in $\om'$, being $\om'$  an open set with $\overline\om\subset\om'$. By Lemma \ref{reynolds} there exists a sequence $(\v_{j})_{j\in\mathbb N}\subset C^1(\om',\R^3)$ such that \eqref{1}, \eqref{2},\eqref{3} and \eqref{4} hold.
Hence,  \eqref{reg}, \eqref{1} and \eqref{4} together with $\mathcal W(x, \mathbf I)=0,\ D\mathcal W(x, \mathbf I)=0$, see \eqref{smooth0}, imply that  
\[
\lim_{j\to+\infty} h_j^{-2} \mathcal W^I(x,\mathbf I+h_j \nabla\v_{j})=\lim_{j\to+\infty} h_j^{-2}\mathcal W(x,\mathbf I+ h_j\nabla\v_{j})=\frac12\,\mathbb E(\v)D^2\mathcal W(x,\mathbf I)\,\mathbb E(\v)
\]
for a.e. $x\in\om$, and that there exists a constant $C'> 0$ such that for $h_j$ small enough there holds
$h^{-2}_j\mathcal W(x,\mathbf I+h_j\nabla\v_j)\le C'|\mathbb E(\v_j)|^2$.
 Therefore by \eqref{3} there exist $q> 1$ and a constant $C''>0$ such that for any large enough $j$
\begin{equation*}
\int_\om\left(\frac1{h_j^2} \mathcal W(x,\mathbf I+h_j \nabla\v_j)\right)^q\, dx\le C'',
\end{equation*}
thus
\begin{equation*}
\displaystyle\lim_{j\to\infty} \int_\om \frac1{h_j^2} \mathcal W^I(x,\mathbf I+h_j \nabla\v_j)\,dx-\mathcal L(\v_j)=\frac12\int_\om \mathbb E(\v)D^2\mathcal W(x,\mathbf I)\,\mathbb E(\v)\,dx-\mathcal L(\v).
\end{equation*}
This shows that $\mathcal G^I_+ (\v)\le \mathcal G^I(\v) $ whenever $\v$ is the restriction to $\om$ of a function $\v\in C^1(\om',\mathbb R^3)$ such that $\v=0$ on $\Gamma$ and $\dv\v=0$ in $\om'$, being $\om'$  an open set with $\overline\om\subset\om'$.

Assume now
that $ \v\in H^{1}(\om,\mathbb R^{3}),\ \dv\, \v=0$ a.e. in $\om$ and $\v=0$ on $\Gamma_\delta$ for some
 $0<\delta < \mu_0$, where $\mu_0$ is defined by \eqref{mu0}.   Then by Lemma \ref{gammal} there exist an open set $\om '$ such that $\overline\om\subset \om ' $ and a sequence
 $(\v_{j})_{j\in\mathbb N}\subset C^1({\om '},\R^3)$ such that $\mathrm{div}\, \v_{j}=0$ a.e. in $\om '$, $\v_{j}=0$ on $\Gamma$ and $\v_{j}\to \v$ in   $H^{1}(\om,\mathbb R^{3})$. Therefore,
\begin{equation*}
\mathcal G^{I}_+ (\v_j)\le \mathcal G^{I}(\v_j). 
\end{equation*}
By taking into account that $\mathcal G^{I}_+$ is weakly lower semicontinuous  in $W^{1,p}(\om,\mathbb R^3)$ and that $\mathcal G^I$ is continuous with respect the strong convergence in $H^1(\om,\mathbb R^3)$ we  get
\begin{equation*}\lab{GammaH1}
\mathcal G^{I}_+ (\v)\le \mathcal G^{I}(\v) 
\end{equation*}
for every $ \v\in H^{1}(\om,\mathbb R^{3})$ such that $ \dv\, \v=0$ a.e. in $\om$ and $\v=0$ on $\Gamma_\delta$ for some
 $0<\delta < \mu_0$. If $\partial\Gamma=\emptyset$, then $\Gamma_\delta=\Gamma$ and the proof is concluded. Suppose instead that $\partial\Gamma\neq\emptyset$
and  let $ \v\in H^{1}(\om,\mathbb R^{3}),\ \dv\, \v=0$ a.e. in $\om$ and $\v=0$ on $\Gamma$. By Lemma \ref{GlvsG} there exist a vanishing sequence $(\lambda_j)_{j\in\mathbb N}\subset(0,\mu_0)$ and a sequence $(\v_j)_{j\in\mathbb N}\subset H^{1}(\om,\mathbb R^{3})$ such that $\dv\, \v_j=0$ a.e. in $\om$,  $\v_j=0$ on $\Gamma_{\lambda_j}$ and $\v_j\to\v$ in $H^{1}(\om,\mathbb R^{3})$. Then
 \begin{equation*}
\mathcal G^{I}_+ (\v_j)\le \mathcal G^{I}(\v_j) 
\end{equation*}
and by exploiting again the weak lower semicontinuity of $\mathcal G^{I}_+$ in $W^{1,p}(\om,\mathbb R^3)$ and continuity of $ \mathcal G^{I}$ in $H^1(\om,\mathbb R^3)$, we achieve the result.
\end{proof}

The proof of the main result directly follows.

\begin{proofad1}
We prove first that $\mathcal G^I$ has a unique minimizer. 
 Weak $H^1(\om,\mathbb R^3)$ compactness of minimizing sequences follows from \eqref{ellipticity} along with Korn and Poincar\'e inequalities.
Along a sequence that converges weakly in $H^1(\om,\mathbb R^3)$, the elastic part of the energy is lower semicontinuous, functional  $\mathcal L$ is continuous, and the divergence-free constraint passes to the limit as well as the vanishing constraint on $\Gamma$. This shows existence of minimizers of $\mathcal G^I$.    
  Let us prove uniqueness of minimizers. Let
\begin{equation}\lab{V0}
\displaystyle\mathcal V_0(x, \mathbf B):=\frac{1}{2}\sym \mathbf B D^2\mathcal W(x,\mathbf I)\sym \mathbf B
\end{equation}
and let $\v_*,\ \v_{**}$ be two minimizers of ${\mathcal G}^I$ (in particular, $\v_*=\v_{**}=0$ on $\Gamma$). Then by first order minimality conditions we have
\begin{equation}\lab{eulerV0}\begin{aligned}
&\displaystyle\int_\om D\mathcal V_0(x,\mathbb E(\v_*))\cdot (\mathbb E(\v_*)-\mathbb E(\v_{**}))\,dx\\
&\qquad=\displaystyle\int_\om D\mathcal V_0(x,\mathbb E(\v_{**}))\cdot (\mathbb E(\v_*)-\mathbb E(\v_{**}))\,dx=\mathcal L(\v_*-\v_{**}).
\end{aligned}
\end{equation}
Hence, by \eqref{V0} and \eqref{eulerV0}
\begin{equation*}
2\int_\om\mathcal V_0(x,\mathbb E(\v_*)-\mathbb E(\v_{**}))=\int_\om D\mathcal V_0(x,\mathbb E(\v_*)-\mathbb E(\v_{**}))\cdot (\mathbb E(\v_*)-\mathbb E(\v_{**}))\,dx=0,
\end{equation*}
therefore \eqref{ellipticity} implies $\mathbb E(\v_*)-\mathbb E(\v_{**})=0$. Since $\v_*-\v_{**}=0$ on $\Gamma$, we deduce $\v_*-\v_{**}=0$ a.e. on  $\om$ thus proving uniqueness. From now we denote by $\v_*$ the unique minimizer of $\overline{\mathcal G}^I$.

By testing with the trivial displacement field, we see that $\inf\mathcal G^I_{h_j}\le 0 $ for any $j\in\mathbb N$. On the other hand, since $\mathcal L(\v)\le C_{\mathcal L}\|\v\|_{W^{1,p}(\om,\mathbb R^3)}$, boundedness from below of functional $\mathcal G^I_{h_j}$ easily follows from \eqref{bsec} and Friedrichs inequality as soon as $j$ is large enough. This proves \eqref{convmin}.

The sequence $(\v_j)_{j\in\mathbb N}$ is bounded in $W^{1,p}(\om,\mathbb R^3)$, thanks to Lemma \ref{add}. Therefore, let us consider a (not relabeled) subsequence such that $\v_j\rightharpoonup\v$ weakly in $W^{1,p}(\om,\mathbb R^3)$. Let $\widetilde \v\in H^1(\om,\mathbb R^3) $ be such that $\widetilde\v=0$ on $\Gamma$ and $\mathrm{div}\,\widetilde\v=0$ a.e. in $\om$. Let $(\widetilde\v_j)_{j\in\mathbb N}\subset W^{1,p}(\om,\mathbb R^3)$ be a recovery sequence for $\widetilde \v$, provided by Lemma \ref{convfunc}.  By taking advantage of \eqref{assinf} and of the $\Gamma$-liminf inequality, still provided by Lemma \ref{convfunc}, we conclude that 
\[
\mathcal G^I(\v)\le\liminf_{j\to+\infty}\mathcal G^I_{h_j}(\v_j)\le\limsup_{j\to+\infty}\mathcal G^I_{h_j}(\widetilde\v_j)=\mathcal G^I(\widetilde \v).
\]  
By the arbitrariness of $\widetilde \v$ we get $\v\in \argmin\mathcal G^I$ hence $\v=\v_*$ and the whole sequence $(\v_j)_{j\in\mathbb N}$ converges to $ \v_*$ weakly in $W^{1,p}(\om,\mathbb R^3)$ thus concluding the proof.
\end{proofad1}

The proof of Corollary 2.2 relies on the following preliminary result.
\begin{lemma}\lab{GbarvsGtilde} { Under the assumptions of {\rm Corollary \ref{nonhom}}},
let $ \overline\v\in W^{1,\infty}(\om,\mathbb R^3)$  be such that $\dv\overline\v=0\,$ a.e. in $\om$.  Then $\overline{\mathcal G}^I$ from \eqref{Gvbar} has a  unique minimizer and if $\v_*\in
\argmin\overline{\mathcal G}^I$ then $\v_*-\overline\v$ is the unique minimizer of
$\widetilde{\mathcal G}^I$ and $\widetilde{\mathcal G}^I(\v_*-\overline\v)=\overline{\mathcal G}^I(\v_*)$, { where $\widetilde{\mathcal G}^I$ is defined by \eqref{lastlabel}}.
\end{lemma}
\begin{proof} Existence of a minimizer of $\overline{\mathcal G}^I$ and of $\widetilde{\mathcal G}^I$  again follows from classical results while regarding uniqueness of minimizers of $\overline{\mathcal G}^I$ and of 
$\widetilde{\mathcal G}^I$ we may argue as in the proof of Theorem \ref{mainth1}
and from now we denote by $\v_*$ the unique minimizer of $\overline{\mathcal G}^I$.

Let  $\u\in H^1_{\mathrm{div}}(\om,\mathbb R^3)$ be such that  $\u=0$ on $\Gamma$,  and set $\v=\u+\overline\v$. Then
$\v=\overline\v$ on $\Gamma$, $\dv\v=0$ a.e. in $\om$ and by using \eqref{V0}
\begin{equation*}\begin{aligned}
&\int_\om\mathcal V_0(x,\mathbb E(\v_*)-\mathbb E(\overline\v))+\int_\om D\mathcal V_0(x,\mathbb E(\overline\v))\cdot (\mathbb E(\v_*)-\mathbb E(\overline\v))\,dx-\mathcal L(\v_*-\overline\v)\\
&\qquad\displaystyle=\int_\om\mathcal V_0(x,\mathbb E(\v_*))\,dx+\int_\om\mathcal V_0(x,\mathbb E(\overline\v))-\int_\om D\mathcal V_0(x,\mathbb E(\overline\v))\cdot\mathbb E(\overline\v)\,dx-\mathcal L(\v_*-\overline\v)\\
&\qquad\displaystyle\le \int_\om\mathcal V_0(x,\mathbb E(\v))\,dx+\int_\om\mathcal V_0(x,\mathbb E(\overline\v))-\int_\om D\mathcal V_0(x,\mathbb E(\overline\v))\cdot\mathbb E(\overline\v)\,dx\\
&\qquad\quad+\int_\om D\mathcal V_0(x,\mathbb E(\overline\v))\cdot\mathbb E(\v)\,dx-\int_\om D\mathcal V_0(x,\mathbb E(\overline\v))\cdot\mathbb E(\v)\,dx-\mathcal L (\v-\overline\v)\\
&\qquad\displaystyle= \int_\om\mathcal V_0(x,\mathbb E(\v)-\mathbb E(\overline\v))\,dx+\int_\om D\mathcal V_0(x,\mathbb E(\overline\v))\cdot(\mathbb E(\v)-\mathbb E(\overline\v))\,dx-\mathcal L (\v-\overline\v),
\end{aligned}
\end{equation*}
that is, $\widetilde{\mathcal G}^I(\mathbf v_*-\overline{\mathbf \v})\le \widetilde{\mathcal G}^I(\mathbf u)$,
thus proving minimality of $\v_*-\overline\v$ for $\widetilde{\mathcal G}^I$ by the arbitrariness of $\mathbf u$. Uniqueness of such a minimizer follows by reasoning as in the first part of this proof so we have only to prove that $\widetilde{\mathcal G}^I(\v_*-\overline\v)=\overline{\mathcal G}^I(\v_*)$. Indeed
\begin{equation*}\begin{aligned}
\displaystyle\widetilde{\mathcal G}^I(\v_*-\overline\v)&=\int_\om \mathcal V_0(x,\mathbb E(\v_*))\,dx+\int_\om \mathcal V_0(x,\mathbb E(\overline\v))\,dx-\int_\om D\mathcal V_0(x,\mathbb E(\overline\v))\cdot\mathbb E(\v_*)\,dx\\
&\qquad\displaystyle-\mathcal L(\v_*-\overline\v)+ \int_\om D\mathcal V_0(x,\mathbb E(\overline\v))\cdot(\mathbb E(\v_*)-\mathbb E(\overline\v))\,dx + \overline{\mathcal G}^I(\overline\v)\\
&\displaystyle=\int_\om \mathcal V_0(x,\mathbb E(\v_*))\,dx-\int_\om \mathcal V_0(x,\mathbb E(\overline\v))\,dx-\mathcal L (\v_*-\overline\v)+\overline{\mathcal G}^I(\overline\v)=\overline{\mathcal G}^I(\v_*)
\end{aligned}
\end{equation*}
and the proof is concluded.
\end{proof}
\begin{proofad2}
Since the map
\begin{equation*}
W^{1,p}(\om,\mathbb R^3)\ni\v\mapsto\int_\om \mathbb E(\overline\v)\, D^2\mathcal W(x,\mathbf I)\,\mathbb E(\v)\,dx
\end{equation*}
is continuous with respect to the weak topology of $W^{1,p}(\om,\mathbb R^3)$,
 Lemma \ref{convfunc} implies the $\Gamma$-convergence of functionals $\widetilde{\mathcal G}_{h_j}^I$ to $\widetilde{\mathcal G}^I$ with respect to the same topology. 

We notice that $\widetilde{\mathcal G}^I_{h_j}(\mathbf 0)\le \overline{\mathcal G}^I(\overline\v)$ so that $\inf \widetilde{\mathcal G}^I_{h_j}<+\infty$ for any $j\in\mathbb N$, where the infimum is taken on $W^{1,p}(\om,\mathbb R^3)$. Since $\overline{\mathcal G}^I(\overline\v) \in\mathbb R$
 and since by assumption \ref{reg} there holds
 \begin{equation}\label{signestimate}\int_\om \mathbb E(\overline\v)\, D^2\mathcal W(x,\mathbf I)\,\mathbb E(\v)\,dx\le K |\om|^{\frac{p-1}{p}} \|\nabla\overline \v\|_{L^\infty(\om,\mathbb R^{3\times3})}\,\|\v\|_{W^{1,p}(\om,\mathbb R^3)},
 \end{equation}
by the same reasoning of the proof of Theorem \ref{mainth1} we deduce boundedness from below of $\widetilde{\mathcal G}^I_{h_j}$ for any large enough $j$ so that \eqref{convmin2} holds.

 Let now  $(\mathbf v_j)_{j\in\mathbb N}\subset H^1(\om,\mathbb R^3)$ be a sequence such that $\v_j=0$ on $\Gamma$ and such that
 \eqref{assinf2} holds. By the same argument of the proof of Lemma \ref{add}, this time also taking \eqref{signestimate} into account, we deduce that $\sup_{j\in\mathbb N}\|\v_j\|_{W^{1,p}(\om,\mathbb R^3)}<+\infty$. 
Therefore, up to not relabeled subsequences, $\v_j\wconv \v_0$ weakly in $W^{1,p}(\om,\mathbb R^3)$. Thanks to $\Gamma$-convergence, by the same argument of the proof of Theorem \ref{mainth1}, we conclude that $  \widetilde{\mathcal G}^{I}_{h_{j}}(\v_{j})\to \widetilde{\mathcal G}^{I}(\v_0)$ as $j\to+\infty$ and that $\v_0$ is the unique minimizer of $\widetilde{\mathcal G}^I$ over $W^{1,p}(\om,\mathbb R^3)$. In particular, the whole sequence $(\v_j)$ converges to $\v_0$ weakly in $W^{1,p}(\om,\mathbb R^3)$. By Lemma \ref{GbarvsGtilde}, $\v_0+\overline\v\in\mathrm{argmin}\,\overline{\mathcal G}^I$ and $\widetilde{\mathcal G}^{I}(\v_0)=\overline{\mathcal G}^I(\v_0+\overline\v)
$ so that we have recovered the unique minimizer of $\overline {\mathcal G}^I$, thus concluding the proof.
\end{proofad2}

\subsection*{Acknowledgements} 
The authors wish to thank Giuseppe Savar\'e for useful suggestions about Lemma \ref{reynolds}.
The authors acknowledge support from the MIUR-PRIN  project  No 2017TEXA3H.
The authors are members of the
GNAMPA group of the Istituto Nazionale di Alta Matematica (INdAM).

\end{document}